\documentclass[12pt]{article}
\usepackage{graphicx}
\pdfoutput=1
\usepackage{amsmath, amssymb, amsthm}

\newcommand{\bR}{\mathbb{R}} 
\newcommand{\bZ}{\mathbb{Z}}

\newcommand{\bV}{\mathbb{V}}

\newcommand{\bE}{\mathbb{E}} 
\newcommand{\eps}{\varepsilon}
\newtheorem{thm}{Theorem}[section]
\newtheorem{asm}{Assumption}[section]

 \newtheorem{cor}[thm]{Corollary}
 \newtheorem{lem}[thm]{Lemma}
 \newtheorem{prop}[thm]{Proposition}
\newtheorem{defn}[thm]{Definition}
 \newtheorem{rem}[thm]{Remark}
\def\vec#1{\mathchoice{\mbox{\boldmath$\displaystyle\bf#1$}}
{\mbox{\boldmath$\textstyle\bf#1$}}
{\mbox{\boldmath$\scriptstyle\bf#1$}}
{\mbox{\boldmath$\scriptscriptstyle\bf#1$}}}
\begin{document}



\title{\bf Robust functional estimation in the multivariate partial linear model}
\date{}
\author{Michael Levine}
\maketitle

\begin{abstract}
We consider the problem of adaptive estimation of the functional component in a multivariate partial linear model where the argument of the function is defined on a $q$-dimensional grid. Obtaining an adaptive estimator of this functional component is an important practical problem in econometrics where exact distributions of random errors and the parametric component are mostly unknown and cannot safely assumed to be normal. An estimator of the functional component that is adaptive in the mean squared sense over the wide range of multivariate Besov classes and robust to a wide choice of distributions of the linear component and random errors is constructed. It is also shown that the same estimator is locally adaptive over the same range of Besov classes and robust over large collections of distributions of the linear component and random errors as well. At any fixed point, this estimator also attains a local adaptive minimax rate. The procedure needed to obtain such an estimator turns out to depend on the choice of the right shrinkage approach in the wavelet domain. We show that one possible approach is to use the multivariate version of the classical BlockJS method. The multivariate version of BlockJS is developed in the manuscript and is shown to represent an independent interest. Finally, the Besov space scale over which the proposed estimator is locally adaptive is shown to depend on the dimensionality of the domain of the functional component; the higher the dimension, the larger the smoothness indicator of Besov spaces must be.  
\end{abstract}
\section{Introduction}\label{intro}

In this manuscript, we consider a partial linear multivariate model defined as  
\begin{equation}\label{model1}
Y_{\vec i}=a+X_{\vec i}^{'}\vec \beta+f(U_{\vec i})+\xi_{\vec i}
\end{equation}
where $X_{\vec i}\in \bR^{p}$ and $U_{\vec i}\in \bR^{q}$, $\beta$ is an unknown $p\times 1$ vector of parameters, $a$ an unknown intercept term, $f(\cdot)$ is an unknown function, and $\xi_{\vec i}$ are independent and identically distributed random variables. No moment conditions are imposed on $\xi_{\vec i}$ but we do assume for convenience that the median of $\xi_{\vec i}$ is equal to zero. $X_{\vec i}$ is defined as a $p$-dimensional continuous random variable. In this manuscript, we consider a special case where each $U_{\vec i}$ is viewed as a $q$-dimensional vector with each coordinate defined on an equispaced grid on $[0,1]$.  The sequence $\{X_{\vec i}\}$ is assumed to be independent of $\{\xi_{\vec i}\}$. We use bold font for indices since  the most convenient notation for this model involves multivariate indices; the detailed description of these indices is postponed until Section \eqref{Section2}. In our manuscript, we only consider the case where $q>1$ that has been relatively little explored in statistical literature. More specifically, the only papers in the statistical literature that we are aware of discussing the multivariate case are \cite{he1996bivariate}, \cite{schick1996root}, \cite{muller2012estimating}, \cite{levine2015minimax}, and \cite{brown2016semiparametric}. Econometric literature discusses the multivariate case to some extent. In particular, \cite{hardle2012partially} contains a review of some possible applications, clearly showing practical utility of considering the case of $q>1$. 

Partial linear models are, in many cases, preferable to purely nonparametric regression model because of the well-known ``curse of dimensionality". For the most part, the parametric part can be estimated at the $\sqrt{n}$ rate where $n$ is the sample size. At the same time, the estimation precision of the nonparametric component usually decays as the dimensionality of its argument grows. Partial linear models have a long history of application in both econometrics and statistics. From the practical viewpoint, they are often of considerable interest  because relationships between the response and predictors in the same model may be of a very different nature. Some of these relationships can often presumed to be linear while others are harder to parameterize. In the most common case, a small subset of the variables are presumed to have an unknown nonlinear relationship with the response while the rest are assumed to have a linear relationship with it. A good example can be found in \cite{schmalensee1999household} that considered a partial linear model to analyze gasoline household consumption in the US. In this model the demand for gasoline, measured in log number of gallons, is assumed to depend linearly on the number of drivers in the family, the household size, residence, region, and lifecycle. At the same time, the response is assumed to depend nonlinearly on the two remaining covariates, log of the household income and log of the age of the head of household.  

Most often, the estimation of the nonparametric component is conducted in order to suggest a possible parametric form for this component where no prior rationale for choosing such a form is available. This, in turn, allows a researcher to describe the data more parsimoniously. For example, the above mentioned \cite{schmalensee1999household} fits the function $g$ that describes the dependence of the log demand for gasoline $Y$ on the log of the household income $Z^{1}$ and the log of the age of the head of household $Z^{2}$. When the function is fit, \cite{schmalensee1999household} plots it first as a function of $Z^{1}$ for several fixed values of $Z^{2}$ and as a function of $Z^{2}$ for several fixed values of $Z^{1}$. The plots suggest that a possible parsimonious representation of the nonparametric component can be a piecewise linear function in both $Z^{1}$ and $Z^{2}$; later diagnostic testing confirms this conclusion.  An in-depth discussion of this issue can be found in \cite{horowitz2009semiparametric}. 

Our main goal in this manuscript is to construct an estimator of the nonparametric component $f$ that is adaptive over a range of functional classes for $f$ and robust with respect to the wide choice of distributions of $X$ and $\xi$. In other words, we want to develop the estimator of the function $f$ that achieves the optimal, or nearly optimal, rate of convergence over a wide range of functional classes for $f$ and that is reasonably robust to choices of a distribution of $X$ and that of random errors $\xi$. To the best of our knowledge, this question has not been considered before in statistical literature. The farthest step in this direction seems to have been made in \cite{brown2016semiparametric} who constructed an asymptotically efficient estimator of $f$ in the model \eqref{model1}. Also, a somewhat related result in the existing literature can be found in \cite{wang2010estimation}. They obtained some optimal convergence rates for a link function in the partial linear {\it single index} problem.  

There are a number of reasons why the estimator robust to the wide choice of possible distributions of $X$ and $\xi$ is of interest. First, typical theory for models of the type \eqref{model1} assumes that errors $\xi_{i}$ are independent, identically distributed (i.i.d.) normal random variables. From the practical viewpoint, normality may not always be satisfactory; see, for example, \cite{stuck1974statistical} and \cite{stuck2000historical}. Moreover, the use of Gaussian formulation often implies that the specifically {\it mean} regression is of interest; in other words, the value of $f(U_{i})$ is viewed as a conditional expectation $\bE\,[Y_{\vec i}|X_{\vec i}]-a-X_{\vec i}^{'}\beta$. In general, however, it is desirable to be able to handle more general formulations, such as median regression and other quantile regressions, as well. This has been noticed as early as \cite{he1996bivariate} who suggested using an M-type objective function in order to treat mean regression, median regression, and other quantile regressions in one setting. Until the present time, if normality has not been required, it has been the usual approach in most statistical and econometric research to impose some moment assumptions that are necessary to obtain asymptotic results for estimators of the parametric component $\vec \beta$; see, e.g. \cite{robinson1988root} and \cite{hardle2012partially}. In this manuscript, we argue that this is also unnecessary and that the estimator robust to the wide choice of possible distributions for $\xi$ and achieving minimax or nearly minimax rate of convergence can be constructed without any moment assumptions being imposed on the distribution of $\xi$.  

Another issue lies in the fact that the distribution of $X$ is often not known in practice; in particular, it need not be a multivariate Gaussian. In econometric practice, in particular, it is exceedingly common to have to deal with a vector $X$ that includes at least some discrete components. An earlier mentioned model \cite{schmalensee1999household}  has the vector $X$ that consists of discrete components only. In general, \cite{horowitz2009semiparametric}, p. $53$ notes when introducing partial linear models in econometric context that ``...$X$ may be discrete or continuous". Thus, from the practical viewpoint it seems inadvisable to limit the distribution of $X$ to the multivariate normal only.  In our manuscript, we develop an estimation method that is adaptive over a wide range of functional classes for $f$ and robust over a large collection of error distributions for $\xi_{\vec i}$ and design vector distributions for $X_{\vec i}$.  This means, in particular, that an exact optimal rate of convergence for both mean squared error risk and the squared error at a point risk is achieved over a wide range of multivariate Besov classes for $f$ {\it  without}  the prior knowledge of the distribution for the design vector $X$ or the random error distribution. 

The method that we propose is based on the idea of treating the sum of the parametric part $X_{\vec i}\vec \beta$ and the random error $\xi_{\vec i}$ as the ``new" random error $\rho_{\vec i}$ and viewing the model \eqref{model1} as a nonparametric regression with unknown random errors. The usefulness of this idea lies in the fact that the resulting nonparametric regression model is somewhat similar to the nonparametric regression with an unknown error distribution considered in \cite{brown2008robust}. The main difference between it and the model of \cite{brown2008robust} is that the argument of the function $f$ now is multivariate while it was univariate in \cite{brown2008robust}. Thus, although the origin of our problem is rather different from that of \cite{brown2008robust}, solutions of the two problems turn out to be connected. To obtain an adaptive estimator of the function $f$, we divide the $q$-dimensional cube $[0,1]^{q}$ into a number of equal volume bins, take the median of observations in each of these bins, and then apply a wavelet based procedure to these local medians together with a bias correction. Out of several possible procedures, we choose a multivariate generalization of the Block JS procedure developed in \cite{cai1999adaptive}. To the best of our knowledge, the multivariate BlockJS procedure has not been properly described before and therefore may be of independent interest. Note that, in general, the multivariate extension of the adaptive estimation procedure described in \cite{brown2008robust} is highly non-trivial. This is due to the necessity of selecting a correct blockwise shrinkage procedure that can be applied to the empirical wavelet coefficients. Our contribution consists of, first, designing a multivariate version of the by now classical BlockJS procedure of \cite{cai1999adaptive}, and, second, of showing that it is the right choice that produces an adaptive and robust estimator of the functional component $f$. 

Our manuscript is organized as follows. In Section 2, we define our proposed procedure exactly and establish several auxiliary results. The asymptotic properties of our procedure are established in Section 3. Section 4 contains some further discussion of our work, while the formal proofs are relegated to the Appendix. 


\section{General approach to estimation of the functional component}\label{Section2}

\subsection{Methodology for adaptive estimation of the functional component}\label{Subsection1}

As mentioned in the introduction, we begin with defining a random variable $\rho_{\vec i}=X_{\vec i}^{'}\vec\beta+\xi_{\vec i}$ and rewriting the model \eqref{model1} as 
\begin{equation}\label{model2prelim}
Y_{\vec i}=a+f(U_{\vec i})+\rho_{\vec i}. 
\end{equation} 
In this form, \eqref{model2prelim} is simply a multivariate nonparametric regression with an unknown error distribution. Note that in this model the intercept $a$ cannot be absorbed in the design matrix $X$ due to identifiability issues; in order to ensure that the model is identifiable, we have to require that an identifiability condition $\int_{[0,1]^{q}} f(u)du=0$ is satisfied. Otherwise, one can add and subtract $\int_{[0,1]^{q}} f(u)du$ to the right hand side of the model with the new constant becoming $a^{'}=a+\int_{[0,1]^{q}} f(u)du$. Without loss of generality, we will adopt the following simpifying assumptions. First, we assume without loss of generality that $a$ is known and equal to $0$. Second, we assume that the vector median of $X_{\vec i}$ is also equal to zero. Indeed, if this is not the case, \cite{brown2016semiparametric} has shown that there exists an estimator $\hat a$ such that $\hat a=O_{p}(n^{-1/2})$ and that the parameter $\beta$ can be estimated at the same rate as well. Therefore, in a general model, we will have an additional term $A:=a+(med\,X_{\vec i})^{'}\beta$ that can be estimated at the same fast parametric rate of convergence and its presence will not influence the optimal rate of convergence of the adaptive estimator of $f$ that we construct. Appropriate remarks will be made in the proofs in mathematical Appendix as well.  Therefore, from now on we will work with the model 
\begin{equation}\label{model2}
Y_{\vec i}=f(U_{\vec i})+\rho_{\vec i} 
\end{equation}
where it is assumed that the median of $\rho_{\vec i}$ is equal to zero. Most of the classical nonparametric regression has been developed under the assumption of independent and identically distributed (i.i.d.) errors. In particular, a variety of smoothing techniques, such as wavelet thresholding techniques, were developed and shown to be highly adaptive in the Gaussian case. When errors are heavy-tailed, these techniques are typically not applicable. \cite{brown2008robust} worked with the model \eqref{model2} when $q=1$ and noticed that, for example, if $\rho_{\vec i}$ is Cauchy distributed, the maximum observation (out of $n$) will be of the order $n$, instead of $\log n$, which is the case when the errors are normally distributed. This invalidates classical denoising approches, such as the wavelet thresholding, and suggests the need for a different take on this problem. 

Similarly to \cite{brown2008robust},  to estimate the function $f$ adaptively, we bin $Y_{\vec i}$ according to the values of coordinates of $U_{\vec i}$. The sample median is then computed within each bin. Now, bin centers can be treated as independent variables in a multivariate nonparametric regression, bin medians being dependent variables. The number of bins has to be chosen in a suitable range; in our case, it turns out that the number of bins $V\asymp n^{3/4}$ where $n$ is the original sample size, is a suitable choice. The resulting model can be viewed as a Gaussian multivariate nonparametric regression and a multivariate version of the BlockJS method of \cite{cai1999adaptive} can be used to obtain an adaptive estimator of the function $f$. To the best of our knowledge, such a generalization of BlockJS method has not been implemented before. The implementation of the proposed procedure is not difficult, since the number of bins can be chosen as a power of $2$. We will show that the resulting estimator enjoys excellent adaptivity properties over a wide range of multivariate Besov balls and is robust over wide ranging sets of distributions of $X$ and $\xi$.  

Before the detailed description of our procedure, it is necessary to specify the exact notation that will be used. For simplicity, we start with values of $U_{\vec i}$ defined on an equispaced grid. More specifically, we define a point $U_{\vec i}=\left(\frac{i_{1}}{m}, \ldots, \frac{i_{q}}{m}\right)^{'} \in \mathbb{R}^{q}$ where each $i_{k}\in \{0,1,\ldots,m\}$, $k=1,\ldots,q$  for some positive $m$. Note that the total sample size is $n=(m+1)^{q}$. This assumption ensures that $m=o(n)$ as $n\rightarrow \infty$.   In the model \eqref{model1}, the multivariate index is $\vec i=(i_{1},\ldots,i_{q})^{'}$. Throughout this article, we will use bold font for all multivariate indices and a regular font for scalar ones.  We will say that two multivariate indices $\vec i^1=(i^1_1,\ldots, i^1_q) \le \vec i_2=(i^2_1,\ldots,i^2_q)$ if $i^1_k\le i^2_k$ for any $k=1, \ldots,q$; the relationship between $\vec i^1$ and $\vec i^2$ is that of partial ordering.  For convenience, we denote a $q$-dimensional vector $\vec n=(m+1,\ldots,m+1)^{'}$ and a $q$-dimensional vector $\vec 0=(0,\ldots,0)^{'}$.  The $l^{2}$ norm of a vector will be denoted $||\cdot||_{2}$. To define the bins we use for aggregating observations, we start with defining $J=\left\lfloor \frac{1}{q}\log_{2}n^{3/4}\right\rfloor$ and $T=2^{J}$. We split the $j$th edge of the $q$-dimensional cube $[0,1]^{q}$ into $T$ intervals of the type $(\frac{l_{j}-1}{T},\frac{l_{j}}{T}]$ where $l_{j}=1,2,\ldots,T$ for any $j=1,2,\ldots,q$. All of the $n$ observations are split into bins in the following way: all of $Y_{\vec i}$ such that the $j$th coordinate of the corresponding $\vec U_{\vec i}$ belongs in an interval $(\frac{l_{j}-1}{T},\frac{l_{j}}{T}]$, $l_{j}=1,2,\ldots,T$,  are assigned to the bin with the multivariate index $\vec l=(l_{1},\ldots,l_{q})^{'}$. We will denote the $\vec l$th bin $D_{\vec l}$. For convenience of notation, we denote a $q$-dimensional vector all of whose coordinates are equal to $\frac{1}{T}$ $\vec {1/T}=(1/T,\ldots,1/T)^{'}$, a $q$-dimensional vector consisting of $1$'s as $\vec 1=(1,\ldots,1)^{'}$, and $\vec T=(T,\ldots,T)^{'}$. We also denote $\frac{\vec l}{T}=(l_{1}/T,\ldots,l_{q}/T)^{'}$. Clearly, the total number of such bins is $V=T^{q}$; note that, due to selection of $J$ and $T$, the total number of bins $V\asymp n^{3/4}$. Also, we define an approximate number of observations in each bin $\kappa=\frac{n}{V}$; clearly, $\kappa\asymp n^{1/4}$.  Let us denote $\eta_{\vec l}$ the median of all $\rho_{\vec i}$ such that the corresponding observation $Y_{\vec i}$ belongs in the $\vec l$th bin $D_{\vec l}$; also, we denote the expectation of the sample median $\eta_{\vec l}$ $b_{\vec l}:=\bE\,\eta_{\vec l}$. 


In order to approximate the median of observations in each multivariate bin with a normal random variable, we will have to develop a multivariate median coupling inequality. What follows is a very brief {\it general} discussion of median coupling; for more details, see \cite{brown2008robust}. Note that our first step will have to be estimation of the expectation of the median $\eta_{\vec l}$ for $\vec l$th bin $D_{\vec l}$. Let $h(x)$ be the density function of $\rho_{\vec i}$, $\vec 0 \le \vec i \le \vec n$. Let $\mathcal{X}_{1},\ldots,\mathcal{X}_{n}$ be independent random variables with the density $h(x)$. We will need the following assumption on the density $h$; this assumption comes from \cite{brown2008robust} but is given here in full for convenience. 
\begin{asm}\label{asmd}
$\int_{-\infty}^{0}h(x)=\frac{1}{2}$, $h(0)>0$, and $h(x)$ is {\it locally} Lipschitz at $x=0$. The Lipschitz condition at zero for $h(x)$ means that there exists a constant $C>0$ such that $\vert h(x)-h(0)\vert \le C\vert x \vert$ in an open neighborhood of $0$. 
\end{asm}
Note that the assumption that $h(0)>0$ guarantees uniqueness of the median of the distribution and the asymptotic normality of the sample median; see e.g. \cite{casella2002statistical} p. $483$.  
In order to approximate a median of the $\vec l$th bin, we will use a coupling inequality of \cite{brown2008robust}. We only give its statement here without a proof.  
\begin{thm}\label{thm_cpl}
Let $\mathcal{X}_{1},\ldots,\mathcal{X}_{n}$ be iid random variables with the density function $h(x)$ that satisfies the Assumption \eqref{asmd} while $Z$ is a standard normal random variable. We assume that $n=2k+1$ for some integer $k\ge 1$. Then, there exists a mapping $\tilde X(Z):\bR\rightarrow \bR$ such that the distribution law of $\tilde X_{med}$  ${\cal L}(\tilde X_{med}(Z))={\cal L}(\mathcal{X}_{med})$ and 
\[
\vert \sqrt{4n}h(0)\tilde X_{med}-Z\vert\le \frac{C}{\sqrt{n}}+\frac{C}{\sqrt{n}}\vert\sqrt{4n}h(0)\tilde X_{med}\vert^{2}
\]
when $\vert\tilde X_{med}\vert \le \eps$ where $C,\eps>0$ depend on the density $h$ but not on $n$. 
\end{thm} 
The bound given in the Theorem \eqref{thm_cpl} can also be expressed in terms of $Z$ as follows. 
\begin{cor}\label{cor_cpl}
Under the assumptions of Theorem \eqref{thm_cpl}, the mapping  $\tilde X_{med}(Z)$ in Theorem \eqref{thm_cpl} also satisfies 
\[
\vert \sqrt{4n}{h(0)}\tilde X_{med}(Z)-Z\vert\le \frac{C}{\sqrt{n}}(1+\vert Z\vert^{2})
\]
when $\vert Z\vert\le \eps\sqrt{n}$ where $C$, $\eps>0$ do not depend on $n$.
\end{cor}

Note that the assumption that the number of observations $n$ is odd has been made for convenience. The Remark $1$ on p. $2059$ of \cite{brown2008robust} shows that it can be dispensed by redefining the median as $\mathcal{X}_{med}=\frac{\mathcal{X}_{(k)}+\mathcal{X}_{(k+1)}}{2}$ when $n=2k$ and following a similar argument. In the future, we will always assume that the number of observations whose median is considered is odd for simplicity. Also, we can say that the Theorem \eqref{thm_cpl} lets us approximate each bin median with a normal random variable that has the mean $f\left(\frac{\vec l}{T}\right)+b_{\vec l}$ and the variance $1/4\kappa h^{2}(0)$.
\begin{rem}
A more general situation that we described briefly earlier would be if the intercept $a$ and/or the median of $X_{\vec i}$ was not equal to zero. If that was the case, we would be looking at approximating the median of all observations $Y_{\vec i}$ from $\vec l$th bin with a normal random variable that has the mean $f\left(\frac{\vec l}{T}\right)+b_{\vec l}+A$ and the variance $1/4\kappa h^{2}(0)$. 
\end{rem}
However, an interesting question remains open. What kind of distributions of $X_{\vec i}$ and $\xi_{\vec i}$ will result in the density $h(x)$ of $\rho_{\vec i}$ that is going to satisfy the Assumption \eqref{asmd}? The following simple proposition identifies a reasonably wide range of distributions of $X_{\vec i}$ and $\xi_{i}$ that will satisfy this Assumption. This proposition takes the form of a convenient sufficient condition. However, before stating it, we need to define a family of elliptical distributions.
\begin{defn}
A random vector $X=(X_{1},\ldots,X_{p})^{'}$ is said to have an elliptical distribution if its characteristic function $\phi(t)$ can be expressed as $\phi(t)=e^{it^{'}\mu}\psi(t^{'}\Sigma t)$ where $\mu$ is the $p\times 1$ vector and $p\times p$ matrix $\Sigma=AA^{'}$ for some matrix $A$. The function $\psi$ is called {\it characteristic generator} of $X$. 
\end{defn}
Some examples of elliptical distributions are normal, t-distribution, Laplace distribution, Cauchy distribution, and many others when $q=1$ and their multivariate analogs when $q>1$. In general, elliptical distributions may not have a density; however, when they do have one, it has the form 
\begin{equation}\label{el.d}
f(X)=\frac{c_{n}}{\sqrt{|\Sigma|}}g_{n}((X-\mu)^{'}\Sigma^{-1}(X-\mu))
\end{equation} where the function $g_{n}$ is called a {\it density generator}, $c_{n}$ is a normalizing constant, and $\mu$ is the median vector (which is equal to the mean if the latter exists). It is a common practice to refer to an elliptical distribution with a density function as $E_{n}(\mu,\Sigma, g_{n})$. The following simple property of elliptical distributions with a density function is easy to establish and so it is given here without a proof for brevity. For detailed discussion, see e.g. \cite{fang1990symmetric}.  
\begin{lem}\label{ell}
Let $X\sim E_{n}(\mu,\Sigma, g_{n})$. Let $B$ be an $r\times q$ matrix while $b\in \bR^{r}$. Then, any affine transformation of $X$ is an elliptical distribution as well:
\[
b+BX\sim E_{r}(b+B\mu,B\Sigma B^{'},g_{r}).
\]
\end{lem}
With the above in mind, we can now state our Proposition. 
\begin{prop}\label{assmp1}
Assume that the density of random errors $\xi_{\vec i}$ $h_{1}(x)$ is Lipschitz continuous at zero, symmetric around zero and strictly positive in an open neighborhood of zero. We also assume that $X_{\vec i}$ has an elliptic distribution with a density function that is continuous in an open neighborhood of zero. Moreover, without loss of generality (as discussed earlier) we assume that the median of $X_{\vec i}$ is zero. Then, the Assumption \eqref{asmd} is satisfied.
\end{prop} 
\begin{proof}
Due to Lemma \eqref{ell}, $X_{\vec i}^{'}\vec \beta$ has a univariate elliptical distribution with the median equal to zero. Clearly, the distribution of $X_{\vec i}^{'}\beta$ also has a density function. The form of the density function of any  elliptical distribution given in \eqref{el.d} clearly implies that, if the median of $X_{\vec i}^{'}\vec \beta$ is zero, the distribution of $X_{\vec i}^{'}\vec \beta$ is symmetric  around zero. Let us denote $h_{3}(x)$ the distribution of $X_{\vec i}^{'}\beta$. Since $X_{\vec i}$ and $\xi_{i}$ are independent, the distribution of $X_{\vec i}^{'}\vec \beta+\xi_{\vec i}$ is a convolution of the two densities: $h(x)=\int_{-\infty}^{\infty}h_{1}(x-u)h_{3}(u)\,du$.  It follows directly from the definition of Lipschitz continuity at zero that, since $h_{1}(x)$ is Lipschitz continuous at zero then so is $h(x)$; moreover, it is not hard to check that the symmetry of $h_{1}(u)$ and $h_{3}(u)$ around zero implies the symmetry of their convolution $h(x)$ around zero as well. Since $h(x)$ is symmetric around zero, we have immediately that $\int_{-\infty}^{0}h(x)=\frac{1}{2}$. Finally, since the density $h_{3}(x)$ is continuous in an open neighborhood of zero and the corresponding median is unique, there exists an open neighborhood of zero where $h_{3}(x)>0$. This, together with strict positivity of $h_{1}(x)$ in some open neighborhood of zero, guarantees that $h(0)=\int_{-\infty}^{\infty}h_{1}(x)h_{3}(x)\,dx>0$.  
\end{proof}
\begin{rem}
Note that Proposition \eqref{assmp1} covers a wide variety of distributions of $X_{\vec i}$ and $\xi_{\vec i}$. Concerning $X_{\vec i}$, distributions such as multivariate normal, multivariate t-distribution, multivariate Laplace, multivariate Cauchy, and, in general, any symmetric stable distribution, are elliptical distributions with a density. It is also of interest that, when the index of stability less than $1$, which is the case for Cauchy and Levy distributions, there is no finite mean yet the resulting distribution of $X_{\vec i}$ still satisfies requirements of the Proposition \eqref{assmp1}. The requirements for the distribution of $\xi_{\vec i}$ also do not contain any moment requirements; note that, for example, a (univariate) Cauchy distribution for $\xi_{\vec i}$ satisfies assumptions of the Proposition \eqref{assmp1}.
\end{rem}

\subsection{Wavelet procedure for a binned semiparametric model}\label{Subsection2}

We will start with a brief introduction into the multivariate wavelet bases. First, let $\{\phi,\psi\}$ be a pair of {\it univariate} father and mother wavelets. We need to assume that both of them are compactly supported on $[0,1]$ and $\int \phi=1$. As is known, translation and dilation of $\phi$ and $\psi$ generates an orthonormal wavelet basis in $L_{2}[0,1]$; thus, in order to obtain a convenient periodized wavelet basis, we have $\phi_{j,k}^{p}(t)=\sum_{l=-\infty}^{\infty}\phi_{j,k}(t-l)$ and $\psi_{j,k}^{p}(t)=\sum_{l=-\infty}^{\infty}\psi_{j,k}(t-l)$ where $\phi_{j,k}(t)=2^{j/2}\phi(2^{j}t-k)$, and $\psi_{j,k}(t)=\sum_{l=-\infty}^{\infty}2^{j/2}\psi(2^{j}t-k)$. The primary resolution level $j_{0}$ should be selected large enough to ensure that the support of the scaling functions (father wavelets) and mother wavelets at level $j_{0}$ does not cover the whole $[0,1]$. We also assume $r$ regularity of our wavelet system for some positive $r$. This means that the first $\lfloor r\rfloor $ moments of the wavelet function $\psi$ are equal to zero. The periodized wavelets generate a curtailed multiresolution ladder:
\[
V_{0}^{p}\subset V_{1}^{p}\subset V_{2}^{p}\subset\cdots
\]
where spaces $V_{j}^{p}$ are spanned by $\phi_{j,k}^{p}(t)$. As in any regular multiresolution analysis, we have $W_{j}^{p}\oplus V_{j}^{p}=V_{j+1}^{p}$ where the space $W_{j}^{p}$ is spanned by $\psi_{j,k}^{p}$. In the future, we will suppress the superscript $p$ from the notation for convenience. An orthonormal wavelet basis has an associated orthogonal Discrete Wavelet Transform (DWT) whose function is to transform sampled data into the wavelet coefficients. A square integrable function $f$ on $[0,1]$ can be expanded into a wavelet series
\[
f(t)=\sum_{k=1}^{2^{j_{0}}}\theta_{j_{0},k}\phi_{j_{0},k}(t)+\sum_{j=j_{0}}^{\infty}\sum_{k=1}^{2^{j}}\theta_{j,k}\psi_{j,k}(t)
\]
where $\theta_{j_{0},k}=\langle f,\phi_{j_{0},k}\rangle$ and $\theta_{j,k}=\langle f,\psi_{j,k}\rangle$ are the wavelet coefficients of $f$. 
We will use a one-dimensional orthogonal wavelet basis to define a $q$-dimensional one. There are a number of ways to construct such a basis based on the one-dimensional one; we will use the one that is based on a tensor product construction and preserves a multiresolution analysis (MRA) in a $q$-dimensional space. By using $q$ univariate orthogonal MRA's
\[
V_{0,(i)}\subset V_{1,(i)}\subset V_{2,(i)}\subset\cdots \subset L_{2}[0,1],
\]
$i=1,2,\ldots,q$, we can define a $q$-dimensional multiresolution analysis
\[
\bV_{0}\subset \bV_{1}\subset \bV_{2}\subset \cdots \cdots L_{2}[0,1]^{q}
\]
in which $\bV_{j}=\otimes_{i=1}^{q}V_{j,(i)} \subset L_{2}[0,1]^{q}$. The resulting $q$-dimensional multiresolution analysis corresponds to, first, one $q$-variate scaling function $\phi(\vec u)\equiv \phi(u_{1},\ldots,u_{q})=\prod_{i=1}^{q}\phi_{(i)}(u_{i})$ where $\phi_{(i)}$ is an $i$th copy of the father wavelet, and, second,  $2^{q}-1$ $q$-variate wavelets
\[
\psi^{i}(\vec u)\equiv \psi^{i}(u_{1},\ldots,u_{q})=\prod_{i=1}^{q}\xi_{(i)}(u_{i}),
\] 
$i=1,2,\ldots,2^{q}-1$ where $\xi_{(i)}$ is either $\phi$ or $\psi$ but not all of $\xi_{(i)}$ are equal to the father wavelet $\phi$. Now, let $\vec k=(k_{1},\ldots,k_{q})\in \bZ^{q}$ be the $k$-dimensional lattice. To complete the description of our notation, we also introduce rescaled and translated versions $\phi_{j_{0},\vec k}(\vec u)=2^{j_{0}q/2}\prod_{m=1}^{q}\phi_{(m)}(2^{j_{0}}u_{m}-k_{m})$ and $\psi_{j,\vec k}^{i}(\vec u)=2^{jq/2}\prod_{m=1}^{q}\xi_{(m)}(2^{j}u_{m}-k_{m})$ where $\xi=\phi$ or $\psi$ but not all $\xi=\phi$. Finally, any function $f\in L_{2}[0,1]^{q}$ can be represented as 
\begin{align}
&f(\vec u)=\sum_{\vec 1 \le \vec k \le \vec{2}^{j_0}}\theta_{j_{0},\vec k}\phi_{j_{0},\vec k}(\vec u)\label{bes_repr}\\
&+\sum_{j\ge j_{0}}\sum_{\vec 1 \le \vec k\le {\vec 2}^{j}}\sum_{i=1}^{2^{q}-1}\theta_{j,\vec k}^{i}\psi_{j,\vec k}^{i}(\vec u)\nonumber
\end{align}
For a more detailed discussion of multivariate wavelet bases see, for example,  \cite{daubechies1992ten} and \cite{vidakovic2009statistical}. 

At this point, we can give a detailed description of our estimator of the functional component $f$. The first step is to bin observations $Y_{\vec i}$ according to values of coordinates of $U_{\vec i}$ as described in Chapter \eqref{Subsection1}. Then, the sample median must be computed within each bin. We will use the notation $g\left(\frac{\vec l}{T}\right)=f\left(\frac{\vec l}{T}\right)+b_{\vec l}$ where $b_{\vec l}$ is the median of errors in the $\vec l$th bin as described earlier. We will denote medians of observations in $\vec l$th bin $Q_{\vec l}$, $\vec 1\le \vec l \le \vec T$. Our next step consists of applying a discrete wavelet transform to all of the medians $Q_{\vec l}$. In the multivariate case that we consider, a discrete wavelet transform is a tensor. After the application of the discrete wavelet transform, we can describe the transformed data as 
\[
U=(y_{j_{0},1},\ldots,y_{j_{0},2^{j_{0}}},y_{j_{0},1}^{1},\ldots,y_{j_{0},2^{j_{0}}}^{1},\ldots,y_{J-1,1}^{2^{q}-1},\ldots,y_{J-1,2^{J-1}}^{2^{q}-1}).
\]
Here, $y_{j_{0},r}$, $r=1,\ldots,2^{j_{0}}$ are the gross structure terms at the lowest resolution level, while $y_{j,k}^{i}$, $i=1,\ldots,2^{q}-1$, $j=j_{0},\ldots,J-1$, $k=1,\ldots,2^{j}$ are empirical wavelet coefficients at level $j$ corresponding to the wavelet $i$ that represent fine structure at scale $2^{j}$. The empirical wavelet coefficients can be written as 
\[
y_{j,\vec k}^{i}=\breve{\theta}_{j,\vec k}^{i}+\eps_{j,\vec k}^{i}+\frac{1}{2h(0)\sqrt{n}}z_{j,\vec k}^{i}+\xi_{j,\vec k}^{i}
\]
where $\breve{\theta}_{j,\vec k}^{i}$ are the discrete wavelet coefficients of $g\left(\frac{\vec l}{T}\right)$, $\eps_{j, \vec k}^{i}$ are deterministic errors, $z_{j,\vec k}^{i}$ are iid $N(0,1)$ and $\xi_{j,\vec k}^{i}$ are stochastic errors. Later, we will show that both deterministic errors  $\eps_{j,\vec k}^{i}$ and stochastic errors $\xi_{j,\vec k}^{i}$ are negligible in a certain sense. Ignoring them, we end up with 
\begin{equation}\label{approx}
y_{j,\vec k}^{i}\approx \breve{\theta}_{j,\vec k}^{i}+\frac{1}{2h(0)\sqrt{n}}z_{j,\vec k}^{i}
\end{equation}
which is essentially an idealized sequence model with the noise level $\sigma=\frac{1}{2h(0)\sqrt{n}}$. As a next step, we propose a multivariate generalization of the BlockJS procedure of \cite{cai1999adaptive} and apply it to the empirical coefficients $y_{j,\vec k}^{i}$ as if they were generated by \eqref{approx}. At each resolution level $j$ the empirical wavelet coefficients $y_{j,\vec k}^{i}$  are grouped into nonoverlapping blocks of length $L$. We define each block as $B_{j,u}^{i}$  consisting of observations $y_{j,\vec k}^{i}$  such that, for any choice of $i$th wavelet and $j$th resolution level, only $(u-1)L+1\le k_{s}\le uL$ are included where $s=1,2,\ldots,q$ and $j=j_{0},\ldots,J-1$.  We also define $S_{j,u,i}^{2}\equiv \sum_{\vec k\in B_{j,u}^{i}}(y_{j,\vec k}^{i})^{2}$ the sum of squared empirical wavelet coefficients included in $u$th block. Let $\hat {h}^{2}(0)$ be an estimator of the squared value of the density function $h$ at the point zero. The following shrinkage rule is then applied to each block $B_{j,u}^{i}$:
\begin{equation}\label{block_s}
\hat \theta_{j,\vec k}^{i}=\left(1-\frac{\lambda_{*}L}{4\hat {h}^{2}(0)nS_{j,u,i}^{2}}\right)_{+}y_{j,\vec k}^{i}
\end{equation}
where $\lambda_{*}$ is the solution of the equation $\lambda_{*}-\log \lambda_{*}=3$ and $4n\hat {h}^{2}(0)$ is present due to the fact that the noise level is equal to $\sigma=\frac{1}{2h(0)\sqrt{n}}$. For the gross structure terms, we define an estimator $\hat{\theta}_{j_{0},\vec k}=y_{j_{0},\vec k}$. Now, we reconstruct the estimate of the function $g$ at the points $\frac{\vec l}{T}$ by applying the inverse discrete wavelet transform to the shrunk empirical wavelet coefficients. In other words, the entire function $g(\vec u)$ can be estimated for any $\vec u=\frac{\vec l}{T}$, $\vec 1 \le \vec l \le \vec T$ as 
\[
\hat g(\vec u)=\sum_{\vec 1 \le \vec k\le {\vec 2}^{j_0}}\hat \theta_{j_{0}, \vec k}\phi_{j_{0},\vec k}(\vec u)+\sum_{j=j_{0}}^{J-1}\sum_{\vec 1 \le \vec k\le \vec{2}^j}\sum_{i=1}^{2^{q}-1}\hat \theta_{j,\vec k}^{i}\psi^{i}_{j,\vec k}(\vec u).
\]
The last remaining step is to estimate the median $b_{\vec l}$ in order to obtain an estimate of the function $f$. The following procedure is employed for that purpose. Recall that the $j$th edge of the $q$-dimensional cube $[0,1]^{q}$ has been split into $T$ intervals; each of these intervals contained $\left\lfloor \frac{m+1}{T}\right\rfloor$ observations. We begin with splitting each of these $T$ intervals for a $j$th edge of a $q$-dimensional cube in two in such a way that the smaller half contains $\left\lfloor \frac{m+1}{2T}\right\rfloor$ observations. Now, we define a set of $q$-dimensional bins in the following way. Each bin in this set consists of smaller halves of all one-dimensional intervals for each of the $q$ dimensions. Note that the cardinality of that set of bins will still be $V=T^{q}$. Next, we define $Q_{\vec l}^{*}$ to be the median of the $\vec l$th ``halfbin" with $Q_{\vec l}$ being the median of all observations from the corresponding ${\vec l}$th ``full" bin. Then, the median expectation is defined to be 
\begin{equation}\label{est_med}
\hat b_{\vec l}=\frac{1}{V}\sum_{\vec 1 \le \vec l\le {\vec T}}[Q_{\vec l}^{*}-Q_{\vec l}].
\end{equation}
This lets us define 
\begin{equation}\label{wave_est}
\hat f_{n}(\vec u)=\hat g(\vec u)-\hat b_{\vec l}=\sum_{\vec 1 \le \vec k \le {\vec 2}^{j_0}}\hat \theta_{j_{0},\vec k}\phi_{j_{0},\vec k}(\vec u)+\sum_{j=j_{0}}^{J-1}\sum_{\vec 1 \le \vec k\le {\vec 2}^j}\sum_{i=1}^{2^{q}-1}\hat \theta_{j,\vec k}^{i}\psi^{i}_{j,\vec k}(\vec u)-\hat b_{\vec l}.
\end{equation}

\section{Adaptivity of the procedure}

We study the theoretical properties of our procedure over the Besov spaces in $\bR^{q}$ and over a suitable class of distributions of $X_{\vec i}$ and $\xi_{\vec i}$ as defined in \eqref{model1}. Besov spaces in $\bR^{q}$ make up a very rich class of spaces that incorporates functions of very significant spatial inhomogeneity and includes H\"{o}lder and Sobolev spaces as special cases. Our discussion of Besov spaces is necessarily brief; for details, see for example \cite{triebel2006theory}. Since {\it multivariate} Besov spaces are not used in the statistical literature very often, we will briefly describe them first. Let $e_{i}=(\delta_{i1},\ldots,\delta_{iq})^{'}$ where $\delta_{ij}=I(i=j)$ is the $i$th unit vector. We define the first difference of the function $f$ in the $i$th direction as $\Delta_{i,h}f(U)=f(U+he_{i})-f(U)$ and the second difference as $\Delta_{i,h}^{2}f(U)=\Delta_{i,h}(\Delta_{i,h}f(U))=f(U+2he_{i})-2f(U+he_{i})+f(U)$. Let $\alpha>0$ be a positive integer and $\mathfrak{a}=\lfloor \alpha\rfloor$. We denote the increment in each direction $i$ $h$ where $\vert h\vert<1$. We also need to define $g_{i,h}=(0,1)^{i-1}\times(0,0\vee(1-2h))\times (0,1)^{q-i}$. The so-called Besov norm {\it in the direction $i$} is then defined as 
\[
\|f\|_{b^{\alpha}_{i,s,t}}=\left(\int_{0}^{1}\vert h\vert^{(\alpha-\mathfrak{a})t-1}\left\|\Delta^{2}_{i,h}\left(\frac{\partial^{\mathfrak a}}{\partial u_{i}^{\mathfrak a}}f\right)\right\|^{t}_{L_{s}(g_{i},h)}\,dh\right)^{1/t}
\]
for $t<\infty$  and 
\[
\|f\|_{b^{\alpha}_{i,s,\infty}}=\sup_{0\le h \le 1}\left\{\vert h\vert^{\alpha-\mathfrak{a}}\left\|\Delta^{2}_{i,h}\left(\frac{\partial^{\mathfrak a}}{\partial u_{i}^{\mathfrak a}}f\right)\right\|_{L_{s}(g_{i},h)}\right\}
\]
otherwise. The Besov norm is now defined as the sum of the $L_{s}$ norm of the function $f$ and Besov norms in all possible directions:
\begin{equation}\label{besov_fnorm}
\|f\|_{B^{\alpha}_{s,t}}=\|f\|_{L_{s}[0,1]^{q}}+\sum_{i=1}^{q}\|f\|_{b^{\alpha}_{i,s,t}}.
\end{equation} A Besov class, sometimes also called a Besov ball,  can be defined as $B^{\alpha}_{s,t}(M)\doteq \{f|\|f\|_{B^{\alpha}_{s,t}}\le K\}$ for some positive $K$. 

Besov norm of a function can also be characterized in terms of its wavelet coefficients. For a fixed primary resolution level $j_{0}$, the Besov sequence norm of the wavelet coefficients of a function $f$ can be defined the following way. First,  let $\theta_{j_{0}}$ be a vector of the father wavelet coefficients at the primary resolution level $j_{0}$, $\theta_{j}$ be the vector of wavelet coefficients at level $j$, and $w=\alpha+q\left(\frac{1}{2}-\frac{1}{s}\right)>0$. Then, we define first
\[
\vert|\theta_{j_{0}}\vert|_{s}=\left(\sum_{\vec k}|\theta_{j_{0},\vec k}|^{s}\right)^{1/s}
\]
and
\[
\vert|\theta_{j}\vert|_{s}=\left(\sum_{\vec k}\sum_{i=1}^{2^{q}-1}\vert\theta_{j,\vec k}^{i}\vert^{s}\right)^{1/s}.
\]

With the above in mind, a sequence norm for a function $f\in B^{\alpha}_{s,t}$ can be defined as as 
\[
\vert|f\vert|_{B^{\alpha}_{s,t}}=\vert|\theta_{j_{0}}\vert|_{s}+\left(\sum_{j=j_{0}}^{\infty}(2^{w}\vert|\theta_{j}\vert|_{s})^{t}\right)^{1/t}.
\]
We know that the Besov sequence norm is equivalent to the Besov function norm defined in \eqref{besov_fnorm} and, therefore, the Besov class can also be defined as $B^{\alpha}_{s,t}(M)=\{f;\|f\|_{B^{\alpha}_{s,t}}\le M\}$; for details, see e.g. \cite{meyer1995wavelets}. It is also known that, in case of Gaussian noise, the minimax risk of estimating $f$ over the Besov body $B^{\alpha}_{s,t}(M)$, 
\begin{equation}
R^{*}(B^{\alpha}_{s,t}(M))=\inf_{\hat f}\sup_{f\in B^{\alpha}_{s,t}}\bE\,\|\hat f-f\|^{2}_{2},
\end{equation}
converges to zero at the rate of $n^{-2\alpha/2\alpha+q}$ as $n\rightarrow \infty$; see e.g. \cite{donoho1995wavelet}. The following theorem shows that our estimator achieves optimal global adaptation for a wide range of multivariate Besov classes $B^{\alpha}_{s,t}(M)$. Moreover, this adaptation is also uniform over a range of distributions of the design vector $X_{\vec i}$ and the error term $\eps_{\vec i}$. To state this theorem properly, we need to define appropriate classes of distributions of $X_{\vec i}$ and $\xi_{\vec i}$. We begin with the following Assumption on the density of $\rho_{\vec i}$ $h(x)$. This Assumption guarantees the existence of some low order moment of $\rho_{\vec i}$. 
\begin{asm}\label{m_asmd}
$\int |x|^{\mathcal{A}}h(x)\,dx<\infty$ for some $\mathcal{A}>0$. 
\end{asm}
For any $0<\eps_{1}<1$, $\eps_{2}>0$, we define the class of densities ${\cal H}_{\eps_{1},\eps_{2}}$ by 
\[
{\cal H}_{\eps_{1},\eps_{2}}=\left\{h:\int_{-\infty}^{0}h(x)\,dx=\frac{1}{2},\eps_{1}\le h(x)\le \frac{1}{\eps_{1}},\vert h(x)-h(0)\vert \le \frac{\vert x\vert}{\eps_{1}} \mbox{ for all } \vert x\vert \le \eps_{2}\right\}.
\]
We will also need another, more narrow class of densities that is defined as ${\cal H}\equiv {\cal H}_{\eps_{1},\eps_{2},\eps_{3},\eps_{4}}$ for some $0<\eps_{1}<1$, $\eps_{i}>0$, $i=2,3,4$ where 
\[
{\cal H}=\left\{h:h\in {\cal H}_{\eps_{1},\eps_{2}}, \vert h^{(3)}(x)\vert \le \eps_{4} \mbox{ for } \vert x\vert \le \eps_{3} \mbox{ and } \int \vert x\vert^{\eps_{3}}h(x)\,dx<\eps_{4}\right\}.
\]
One simple way to ensure that $h(x)\in {\cal H}$ is to require, first, that  densities $h_{1}(u)$ and $h_{2}(u)$ satisfy assumptions of the Proposition \eqref{assmp1}. In addition, we also have to require the existence of a finite moment of some order $\lambda>0$ for the density $h_{1}(u)$ and for the distribution of $X_{\vec i}$ plus the existence of the finite third derivative for $h_{1}(u)$ in a small open neighborhood of zero. Then, the moment assumption is satisfied because, for any $0<\pi\le 1$, and random variables $X_{1},\ldots,X_{n}$, we have $\bE|X_{1}+\cdots+X_{n}|^{\pi}\le \bE\,|X_{1}|^{\pi}+\cdots+\bE\,|X_{n}|^{\pi}$.  Moreover, differentiability of $h_{1}(u)$ ensures the same property for the convolution $h(u)$. With the above in mind, we can define the class ${\cal G}_{1}=\{h_{1}(u): h_{1}(u) \mbox{ symmetric around zero }, h_{1}(u) \mbox{ is Lipschitz at } u=0, h_{1}(u)>\eps_{1} \mbox{ for all } |u|<\eps_{2}, |h^{(3)}_{1}(u)|\le \eps_{3} \mbox { for all } |u|\le \eps_{4}, \int |u|^{\eps_{3}}h_{1}(u)\,du <\eps_{4}\}$. As for $X_{\vec i}$, we define ${\cal G}_{2}$ as a class of all $p$-dimensional elliptical distributions with finite moments of a small order $\eps_{3}>0$. With these definitions in mind, we can now formulate the global adaptivity result. 

\begin{thm}\label{gl_adapt}
 Suppose the wavelet $\psi$ is $r$-regular. Define $d=\min\left(\alpha-\frac{q}{s},1\right)$. Then the estimator $\hat f_{n}$ defined in \eqref{wave_est} satisfies, for $s\ge 2$, $\alpha\le r$, and $\frac{3d}{2q}>\frac{2\alpha}{2\alpha+q}$,
 \[
 \sup_{h_{1}\in {\cal G}_{1},h_{2}\in {\cal G}_{2}}\sup_{f\in B^{\alpha}_{s,t}(M)}\bE\,\|\hat f_{n}-f\|^{2}_{2}\le Cn^{-2\alpha/2\alpha+q},
 \]
 and for $1\le s <2$, $\alpha\le r$ and $\frac{3d}{2q}>\frac{2\alpha}{2\alpha+q}$, 
 \[
 \sup_{h_{1} \in {\cal G}_{1},h_{2}\in {\cal G}_{2}}\sup_{f\in B^{\alpha}_{s,t}(M)}\bE\,\|\hat f_{n}-f\|^{2}_{2}\le Cn^{-2\alpha/2\alpha+q}(\log n)^{2-s/[s(2\alpha+q)+2(1-q)]}.
 \]
 \end{thm}
Effectively, we show that the estimator attains the optimal rate of convergence over a wide range of Besov classes for $f$ and a large collection of the unknown error distributions for $\eps_{i}$, as well as for a large collection of distributions of the design vector $X_{\vec i}$. Note also that when $q=1$, these rates are reduced to those in the one dimensional case, $n^{-2\alpha/2\alpha+1}$ and $n^{-2\alpha/2\alpha+1}(\log n)^{2-s/[s(2\alpha+1)}$, respectively; see, e.g. \cite{brown2008robust}. 

Since functions belonging to Besov classes $B^{\alpha}_{s,t}(M)$ are highly spatially inhomogeneous, local adaptivity at an arbitrary point $\vec u_{0}\in [0,1]^{q}$ should also be investigated. To measure such a spatial adaptivity, the local mean squared risk 
\begin{equation}
R(\hat f_{n}(\vec u_{0}), f(\vec u_{0}))\equiv \bE(\hat f_{n}(\vec u_{0})-f(\vec u_{0}))^{2}
\end{equation}
is used. The precise way of measuring the local smoothness of the function $f(\vec u)$ at a given point $\vec u=\vec u_{0}$ is by the use of its local H\"{o}lder smoothness index, that is a characteristic of the point $\vec u_{0}$. From the technical viewpoint, it is more straightforward to assume that the function $f$ on the {\it entire} cube $[0,1]^{q}$ belongs to a Lipschitz class $\Lambda^{\alpha}(M)$ with the same smoothness index $\alpha$ for every point. To define such a class, we define first for a multivariate index $\vec i$ $|\vec i|=i_{1}+\ldots+i_{q}$. Next, let us select a constant $M>0$, and, for a $q$-dimensional index ${\bf i}=(i_{1},\ldots, i_{q})$, define ${\bf i}(l)=\{{\bf i}:|{\bf i}|=i_{1}+\ldots+i_{q}=l\}$. Then, for any function $f:\mathbb{R}^{q}\rightarrow \mathbb{R}$, the mixed partial derivative of order $l$, $\frac{D^{{\bf i}(l)}f}{{\partial u_{1}^{i_{1}}}\ldots {\partial u_{q}^{i_{q}}}}$ is defined for all ${\bf i}$ such that $|{\bf i}|=l$. With partial derivative thus defined, the Lipschitz class $\Lambda^{\alpha}(M)$ consists of all functions $f(\vec u):[0,1]^{q}\rightarrow \mathbb{R}$ such that $|D^{{\bf i}(l)}f(\vec u)|\le M$ for $l=0,1,\ldots,\lfloor \alpha \rfloor$ and $|D^{{\bf i}(\lfloor \alpha \rfloor)}f(v)-D^{{\bf i}(\lfloor \alpha \rfloor)}f(w)|\le M||v-w||^{\alpha^{'}}$ with $\alpha^{'}=\alpha-\lfloor \alpha \rfloor$. The H\"{o}lder smoothness index $\alpha$ will be used to measure local smoothness of the function $f$. Then, the following theorem shows that our estimator achieves optimal local adaptation uniformly over the same families of distributions of $\vec X_{i}$ and $\eps_{i}$ as before. 
\begin{thm}\label{lc_adapt}
Assume that $f\in \Lambda^{\alpha}(M)$ on $[0,1]^{q}$. Suppose the wavelet $\psi$ is $r$-regular, $r\ge \alpha>\frac{q}{6}$, and $\vec u_{0}\in (0,1)^{q}$ is a fixed point. Then, the estimator $\hat f_{n}$ defined in \eqref{wave_est}, satisfies
\[
\sup_{h_{1}\in {\cal G}_{1},h_{2}\in {\cal G}_{2}}\sup_{f\in \Lambda^{\alpha}(M)}\bE\,\|\hat f_{n}(\vec u_{0})-f(\vec u_{0})\|^{2}_{2}\le C\left(\frac{\log n}{n}\right)^{2\alpha/2\alpha+q}.
\]
\end{thm}
\begin{rem}
Note that the smoothness index of the Besov space scale that we used must remain above the ratio $\frac{q}{6}$. This can be explained by the additional difficulty of uniform estimation over a wide ranging scale of Besov spaces in higher dimensions. In other words, the higher the dimension is, the smoother must the functional scale remain to enable a uniform in mean squared error risk estimation over it.
\end{rem}

\section{Discussion and a future research}

We constructed an estimator of the functional component $f$ in the multivariate partial linear model \eqref{model1} under relatively simple conditions to stress the basic ideas. For example, the assumption of equispaced grid on each edge of the $q$-dimensional cube $[0,1]^{q}$ (standard fixed design) reflects the situation where the data points $U_{\vec i}$ are obtained experimentally.  A more realistic assumption would be to allow for non-equispaced grid; our approach, however, can be modified to account for such a possibility. 

A simple option would be to ignore the fact that the grid is non-equispaced, and keep using the same binning approach as we defined earlier. In this case, bins will have different numbers of observations and the resulting medians will have different variances. This implies that a multivariate wavelet smoothing procedure that accounts for heteroskedasticity will have to be used. To the best of our knowledge, such a procedure has been proposed so far only in the univariate case in \cite{kovac2000extending}. However, a generalization of their procedure to the multivariate case is fairly straightforward. The other possible approach to this problem would be to bin observations in such a way that each bin contains the same number of observations; this would imply that, for each edge of the cube $[0,1]^{q}$, it will be necessary to adjust length of intervals according to the density of the design. This will result in irregularly spaced but homoskedastic data. There are a number of approaches that have been proposed for handling of irregularly spaced data in the wavelet shrinkage context; some of the better known ones are \cite{hall1997interpolation}, \cite{antoniadis1998wavelet}, \cite{cai1998wavelet}, \cite{sardy1999wavelet}, \cite{pensky2001non}, \cite{brown2002asymptotic}, \cite{zhang2002wavelet}, and \cite{amato2006wavelet}. Almost all of them treat the univariate case only; however, a multivariate wavelet thresholding procedure that can adapt to local changes in the smoothness of the regression function and to the distribution of the design has been proposed in \cite{kohler2008multivariate}. That procedure was proposed for a multivariate {\it random} design. However, since the true density of the design is not usually known, in either case one can simply construct an orthonormal wavelet basis of the space $L_{2}[0,1]^{q}(\mu_{n})$ where $\mu_{n}$ is the measure corresponding to the empirical distribution function of $U_{\vec i}$. 

Another issue of importance is the fact that the multivariate curve estimation itself is sometimes regarded as problematic since, even for relatively small number of dimensions $q$, the ``curse of dimensionality" begins to manifest itself and minimax rates of convergence becomes unacceptably slow. In practice, however, the true complexity of the multivariate curve may be considerably lower simply due to the fact that the {\it real} multivariate data are not often truly isotropic. As an example, \cite{scott2015multivariate} remarks in Chapter $2$ that ``Multivariate data in $\bR^{d}$ are almost never $d$-dimensional. That is, the {\it underlying} structure of data in $\bR^{d}$ is almost always of dimensions lower than $d$".  This would imply that it makes sense to model the function $f$ in our model $\eqref{model1}$ as a member of some {\it anisotropic} functional class, e.g. the anisotropic Besov classes $B^{\vec \alpha}_{\vec s,t}(M)$ where parameters $\vec \alpha$ and $\vec s$ are now both $q$-dimensional vectors. Such classes were rigorously defined in \cite{besov1979integral} and the multivariate wavelet thresholding procedure for a multivariate anisotropic Gaussian white noise model was proposed in \cite{neumann2000multivariate}; for a more recent approach to this topic see also \cite{autin2014hyperbolic}. The next sensible step is to allow for a function $f$ to belong to an anisotropic functional class and to obtain an estimator of $f$ that is, again, adaptive to a range of such functional classes and robust with respect to a range of distributions of the design vector $X_{\vec i}$ and of the distributions of random errors $\xi_{\vec i}$. The results of this ongoing work will be reported elsewhere. 

\section{Appendix}

In order to prove our adaptation results, we need to start with the following Proposition that provides us with an expansion of the median of binned observations. That expansion shows that, up to small stochastic and deterministic errors, that median has an approximately normal distribution. An important fact that we need to use in order to prove this Proposition is that any Besov ball $B^{\alpha}_{s,t}(M)$ can be embedded into a H\"{o}lder ball with the smoothness index $d=\min\left(\alpha-\frac{q}{s},1\right)$; for details see e.g. \cite{meyer1995wavelets}. We also remark here that, in this section, we use the notation $C$ for a generic positive constant that can be different from one line to another. 
\begin{prop}\label{med_exp}
Let the function $f\in B^{\alpha}_{s,t}(M)$ and $d\doteq \min\left(\alpha-\frac{q}{s},1\right)$. The median $Q_{\vec l}$ of observations that belong to the $\vec l$th bin can be written as 
\[
\sqrt{\kappa}Q_{\vec l}=\sqrt{\kappa}f\left(\frac{\vec l}{T}\right)+\sqrt{\kappa}b_{\vec l}+\frac{1}{2}Z_{\vec l}+\eps_{\vec l}+\zeta_{\vec l}
\]
where \begin{enumerate}
\item $Z_{\vec l}\sim N(0,1/h^{2}(0));$
\item $\eps_{\vec l}$ are constants such that $\vert \eps_{\vec l}\vert\le C\sqrt{\kappa}q^{d/2}T^{-d};$
\item$\zeta_{\vec l}$ are independent random variables such that for any $r>0$ 
\[
\bE\,\vert \zeta_{\vec l}\vert^{r} \le C_{r}\kappa^{-r/2}+C_{r}\kappa^{r/2}q^{dr/2}T^{-dr}
\]
where $C_{r}$ is a positive constant that depends on $r$ only; moreover, for any $a>0$
\begin{equation}\label{prineq}
P(\vert \zeta_{\vec l}\vert>a)\le C_{r}(a^{2}\kappa)^{-r/2}+C_{r}(a^{2}T^{2d}/\kappa q^{d})^{-r/2}. 
\end{equation}
\end{enumerate}
\end{prop}
 
\begin{proof}
In this proof, we denote the cdf of the standard normal distribution $\Phi(\cdot)$. Define $Z_{\vec l}=\frac{1}{h(0)}\Phi^{-1}(G(\eta_{\vec l}))$ where $G$ is the distribution of the median $\eta_{\vec l}$. Due to Theorem \eqref{thm_cpl}, we know that the rescaled median of errors $\sqrt{4\kappa}\eta_{\vec l}$ can be well approximated by a mean zero normal random variable with the variance equal to $\frac{1}{h^{2}(0)}$. Next, we  define 
\begin{align*}
&\eps_{\vec l}=\sqrt{\kappa}\bE\,Q_{\vec l}-\sqrt{\kappa}f\left(\frac{\vec l}{T}\right)-\sqrt{\kappa}b_{\vec l}\\
&=\bE\,\left\{\sqrt{\kappa}Q_{\vec l}-\sqrt{\kappa}f\left(\frac{\vec l}{T}\right)-\sqrt{\kappa}\eta_{\vec l}\right\}.
\end{align*}
What we have in the above is the deterministic component of the approximation error due to binning. Clearly, for the $\vec l$th bin $D_{\vec l}$, we have 
\begin{align}\label{max}
&\min_{u_{\vec i}\in D_{\vec l}}\left[f(u_{\vec i})-f\left(\frac{\vec l}{T}\right)\right]\\
&\le Q_{\vec l}-\eta_{\vec l}-f\left(\frac{\vec l}{T}\right)\le \max_{u_{\vec i}\in D_{\vec l}}\left[f(u_{\vec i})-f\left(\frac{\vec l}{T}\right)\right]\nonumber
\end{align}
Since the function $f$ is in a H\"{o}lder ball with the smoothness index $d=\min\left(\alpha-\frac{q}{s},1\right)$,  we have 
\begin{align*}
&\vert \eps_{\vec l}\vert \le \sqrt{\kappa}\bE\,\left\vert Q_{\vec l}-f\left(\frac{\vec l}{T}\right)-\eta_{\vec l}\right\vert \\
&\le \sqrt{\kappa}\max_{u_{\vec i}\in D_{\vec l}}\left\vert f(u_{\vec i})-f\left(\frac{\vec l}{T}\right)\right\vert \le C\sqrt{\kappa}q^{d/2}T^{-d}
\end{align*}

Now, it becomes necessary to characterize the random error of our approximation. First, we define $\zeta_{\vec l}=\sqrt{\kappa}Q_{\vec l}-\sqrt{\kappa}f\left(\frac{\vec l}{T}\right)-\sqrt{\kappa}b_{\vec l}-\eps_{\vec l}-\frac{1}{2}Z_{\vec l}$. Note that $\bE\,\zeta_{\vec l}=0$ and this random error can be represented as the sum of two components, $\zeta_{1\vec l}=\sqrt{\kappa}Q_{\vec l}-\sqrt{\kappa}f\left(\frac{\vec l}{T}\right)-\sqrt{\kappa}\eta_{\vec l}-\eps_{\vec l}$ and $\zeta_{2\vec l}=\sqrt{\kappa}\eta_{\vec l}-\sqrt{\kappa}b_{\vec l}-\frac{1}{2}Z_{\vec l}$. The first component $\zeta_{1\vec l}$ represents the random error resulting from the binning of observations while $\zeta_{2\vec l}$ is the error resulting from approximation of the median of random errors with a normal random variable. First, the error $\zeta_{1\vec l}$ is bounded due to \eqref{max} as $\vert \zeta_{1\vec l}\vert \le C\sqrt{\kappa}q^{d/2}T^{-d}$. Next, using the Corollary \eqref{cor_cpl}, we can bound the absolute value of the second random error term as $\vert \zeta_{2\vec l}\vert \le \frac{C}{\kappa^{1/2}}(1+\vert Z_{\vec l}\vert^{2})$ when $\vert Z_{\vec l}\vert \le \eps\sqrt{\kappa}$ for some $\eps>0$.  Thus, for any fixed $r\ge 0$, 
\begin{align*}
&\bE\,\vert \zeta_{2\vec l}\vert^{r}=\bE\,\vert \zeta_{2\vec l}\vert^{r}I(\vert Z_{\vec l}\vert \le \eps\sqrt{\kappa})+\bE\,\vert \zeta_{2\vec l}\vert^{r}I(\vert Z_{\vec l}\vert > \eps\sqrt{\kappa})\\
&\le C\kappa^{-r/2}+\bE\,\vert \zeta_{2\vec l}\vert^{r}I(\vert Z_{\vec l}\vert > \eps\sqrt{\kappa})\\
&\le C\kappa^{-r/2}+\left\{\bE\,\vert \zeta_{2\vec l}\vert^{2r}\right\}^{r/2}\exp\left(-\frac{\eps^{2}}{\kappa}\right)
\end{align*}
since, due to the Mill's ratio inequality, we can quickly verify that $P(\vert Z_{\vec l}\vert > \eps\sqrt{\kappa})\le \exp\left(-\frac{\eps^{2}}{\kappa}\right)$.  

To continue, we will need to use the Assumption (A2). Using the same argument as in \cite{brown2008robust}, we can show that the density of the centered median $\eta_{\vec l}-b_{\vec l}$ $g(x)$ is such that the sample centered median has any finite moments and, therefore, $\bE\,\vert\sqrt{\kappa}(\eta_{\vec l}-b_{\vec l})\vert^{2r}\le \kappa^{r}\bE\,\vert\eta_{\vec l}\vert^{2r}\le D_{r}\kappa^{r}$ for some positive constant $D_{r}$ that does not depend on $n$. This allows us to conclude that  $\bE\,\vert \zeta_{2\vec l}\vert^{r}\le C_{r}\kappa^{-r/2}$ since the normal random variable $Z_{j}$ has finite moments of any order. Finally,  
\begin{align*}
&\bE\,\vert \zeta_{\vec l}\vert^{r}\le 2^{r-1}\left(\bE\,\vert \zeta_{1\vec l}\vert^{r}+\bE\,\vert \zeta_{2\vec l}\vert^{r}\right)\\
&\le C_{r}\kappa^{-r/2}+C_{\vec r}\kappa^{r/2}q^{dr/2}T^{-dr}.
\end{align*} 
The inequality \eqref{prineq} can now be obtained immediately using the Markov's inequality. 
\end{proof}
\begin{rem}
In the proof of our adaptation results, we will assume everywhere that $h(0)$ is known and equal to $1$. This can be done because $h^{-2}(0)$ can always be estimated in such a way that the difference between $h^{-2}(0)$ and $\hat h^{-2}(0)$ is bounded in probability as $O_{p}(n^{-\delta})$  for some $\delta>0$ and, moreover, $P(|\hat h^{-2}(0)-h^{-2}(0)|\ge n^{-\delta})=c_{l}n^{-l}$ for any $l \ge 1$. Such an estimator can be constructed as a properly normalized sum of ordered squared difference of medians of observations $Q_{\vec l}$, $\vec 1 \le \vec l \le V$; to order medians, one can use, for example, a lexicographical order of their indices $\vec l$.. To make the notation easier, we denote two successive medians $Q_{2k-1}$ and $Q_{2k}$, using a scalar index to avoid confusion here. Then, the needed estimator of $h^{-2}(0)$ will be proportional to $\sum_{k}(Q_{2k-1}-Q_{2k})^{2}$. When such an estimator is constructed, one can immediately check that everywhere in proofs asymptotic properties do not change if $\lambda^{*}(1+O(n^{-\delta}))$ is used instead of $\lambda^{*}$. The details are similar to the argument of \cite{brown2008robust} and are omitted for conciseness.  
\end{rem}
The next proposition is needed to obtain a uniform bound of the mean squared error risk of estimating the expected error median. Its proof is very similar to that of Lemma $5$ in \cite{brown2008robust} and is omitted for brevity.  
\begin{prop}\label{med_prop}
Let the expectation of the error median for $\vec l$th bin and its estimate be $b_{\vec l}$ and $\hat b_{\vec l}$, as defined earlier in \eqref{est_med}. Then, 
\[
\sup_{h\in {\cal H}}\left\vert b_{\vec l}+\frac{h^{'}(0)}{8h^{3}(0)\kappa}\right\vert \le C\kappa^{-2},
\]
and
\[
\sup_{h\in {\cal H}}\sup_{f\in B^{\alpha}_{s,t}(M)}\bE\,(\hat b_{\vec l}-b_{\vec l})^{2}\le C\max\{q^{d}T^{-2d},\kappa^{-4}\}
\]
for any index $\vec l$. 
\end{prop} 

\begin{proof}
Without loss of generality, assume that $\kappa=2\nu+1$. Then, the expectation of the $\vec l$th median is
\[
\bE\,\eta_{\vec l}=\int x\frac{2\nu+1}{(\nu!)^{2}}H^{\nu}(x)[1-H(x)]^{\nu}\,dH(x)
\]
where $H(x)$ is the distribution function corresponding to $h(x)$. For any $\delta>0$, define a set $A_{\delta}=\{x:\left\vert H(x)-\frac{1}{2}\right\vert\le \delta\}$. It follows from the definition of the class ${\cal H}$  that there exists a constant $\delta>0$ such that for some $\eps>0$ we have $h^{(3)}(x)\le \frac{1}{\eps}$ and $\eps \le h(x)\le \frac{1}{\eps}$ for any $x\in A_{\delta}$ uniformly for all $h\in {\cal H}$. This property implies that $H^{-1}(x)$ is well defined and is differentiable up to the fourth order for any $x\in A_{\delta}$. Now we can expand the expectation of the median into two parts:
\[
\bE\,\eta_{\vec l}=\left(\int_{A_{\delta}}+\int_{A_{\delta}^{c}}\right)x\frac{2\nu+1}{(\nu!)^{2}}H^{\nu}(x)[1-H(x)]^{\nu}\,dH(x)
\]
Since we earlier established that all of the moments of the median are finite, $Q_{2}$ goes to zero exponentially as $\mu\rightarrow \infty$. Next, we find that 
\begin{align*}
&Q_{1}=\int_{1/2-\delta}^{1/2+\delta}\left(H^{-1}(x)-H^{-1}\left(\frac{1}{2}\right)\right)\frac{(2\nu+1)!}{(\nu!)^{2}}x^{\nu}(1-x)^{\nu}\,dx\\
&=\int_{1/2-\delta}^{1/2+\delta}\left[\frac{1}{2}(H^{-1})^{''}\left(\frac{1}{2}\right)\left(x-\frac{1}{2}\right)^{2}+\frac{(H^{-1})^{(4)}(\tau)}{24}\left(x-\frac{1}{2}\right)^{4}\right]\\
&\times \frac{2\nu+1}{(\nu!)^{2}}x^{\nu}(1-x)^{\nu}\,dx
\end{align*}
since $x^{\nu}(1-x)^{\nu}$ is symmetric around $\frac{1}{2}$. The expression $\frac{2\nu+1}{(\nu!)^{2}}x^{\nu}(1-x)^{\nu}$ is the density function of the Beta$(\nu+1,\nu+1)$ with the mean equal to $\frac{1}{2}$; this, and the fact that $(H^{-1})^{(4)}(\tau)$ is bounded uniformly for all $h\in {\cal H}$ , implies that, in the same way as in \cite{brown2008robust}, $Q_{1}=-\frac{h^{'}(0)}{8h^{3}(0)\kappa}+O\left(\frac{1}{\kappa^{2}}\right)$. 
Recall that we earlier established that (see Proposition \eqref{med_exp}) that 
\[
Q_{\vec l}=f\left(\frac{\vec l}{T}\right)+b_{\vec l}+\frac{1}{2\sqrt{\kappa}}Z_{\vec l}+\frac{1}{\sqrt{\kappa}}\eps_{\vec l}+\frac{1}{\sqrt{\kappa}}\zeta_{\vec l}.
\]
In a similar way, we can write for the median of the ``half" $\vec l$ th bin that 
\[
Q_{\vec l}^{*}=f\left(\frac{\vec l-\vec{ 1/2}}{T}\right)+b_{\vec l}^{*}+\frac{1}{2\sqrt{\nu}}Z_{\vec l}^{*}+\frac{1}{\sqrt{\nu}}\eps_{\vec l}^{*}+\frac{1}{\sqrt{\nu}}\zeta_{\vec l}^{*}
\]
where $\vec {\frac{1}{2}}$ is a $q$-dimensional vector $\left(\frac{1}{2},\ldots,\frac{1}{2}\right)^{'}$, $b_{\vec l}^{*}$ is the  expected median of the errors of all observations in the ``half" $\vec l$th bin, and where $Z_{\vec l}^{*}$, $\eps_{\vec l}^{*}$, and $\zeta_{\vec l}^{*}$ satisfy the Proposition \eqref{med_exp}. Then, the error from the median estimation $\hat b_{\vec l}-b_{\vec l}=\frac{1}{V}\sum_{\vec l}(Q_{\vec l}^{*}-Q_{\vec l})-b_{\vec l}$ can be written as follows:
\begin{align*}
&\hat b_{\vec l}-b_{\vec l}=\frac{1}{V}\sum_{\vec l}\left(f\left(\frac{\vec l-\vec{1/2}}{T}\right)-f\left(\frac{\vec l}{T}\right)\right)+(b_{\vec l}^{*}-2b_{\vec l})\\
&+\left[\frac{1}{\sqrt{\nu}}\frac{1}{V}\sum_{\vec l}\eps_{\vec l}^{*}-\frac{1}{\sqrt{\kappa}}\frac{1}{V}\sum_{\vec l}\eps_{\vec l}\right]\\
&+\left[\frac{1}{2\sqrt{\nu}}\frac{1}{V}\sum_{\vec l}Z_{\vec l}^{*}-\frac{1}{2\sqrt{\kappa}}\frac{1}{V}\sum_{\vec l}Z_{\vec l}\right]\\
&+\left[\frac{1}{\sqrt{\nu}}\frac{1}{V}\sum_{\vec l}\zeta_{\vec l}^{*}-\frac{1}{\sqrt{\kappa}}\frac{1}{V}\sum_{\vec l}\zeta_{\vec l}\right]\\
&\equiv R_{1}+R_{2}+R_{3}+R_{4}+R_{5}.
\end{align*}
Due to the embedding of the Besov ball $B_{s,t}^{\alpha}(M)$ into the H\"{o}lder ball with the smoothness index $d=\min\left(\alpha-\frac{q}{s},1\right)$, the first term is uniformly bounded: $\sup_{f\in B^{\alpha}_{s,t}(M)}R_{1}^{2}\le CT^{-2d}$. The second term is bounded as $\sup_{h\in {\cal H}}R_{2}^{2}\le C\kappa^{-4}$. By Proposition \eqref{med_exp}, the third term is also bounded as $\sup_{h\in {\cal H}, f\in B^{\alpha}_{s,t}(M)}R_{3}^{2}\le CT^{-d}$. Since $Z_{\vec l}^{*}-Z_{\vec l}$ are always independent, we have $\bE\,R_{4}^{2}\le \frac{1}{h^{2}(0)}\left(\frac{1}{\kappa}+\frac{1}{\nu}\right)\frac{1}{T}\le Cn^{-1}$.  Finally, by the Proposition \eqref{med_exp}, we have $\bE\,R_{5}^{2}=o(n^{-1})$. Thus, the overall bound is 
\[
\sup_{h\in {\cal H},f\in B^{\alpha}_{s,t}(M)}\bE\,(\hat b_{\vec l}-b_{\vec l})^{2}\le \max\{T^{-2d},\kappa^{-4}\}.
\] 
\end{proof}

 Due to the Proposition \eqref{med_exp}, we can write
\begin{equation}\label{med_model}
\frac{1}{\sqrt{V}}Q_{\vec l}=\frac{g(\vec l/T)}{\sqrt{V}}+\frac{\eps_{\vec l}}{\sqrt{n}}+\frac{Z_{\vec l}}{2\sqrt{n}}+\frac{\zeta_{\vec l}}{\sqrt{n}}.
\end{equation}

Let $Q$ be the vector of all bin medians $Q_{\vec l}$; such a vector will have the length $V$. Applying the discrete wavelet transform to both sides of \eqref{med_model}, we can expand the empirical wavelet coefficients $y^{i}_{j,\vec k}$ as
\begin{equation}\label{emp_wav}
y^{i}_{j,\vec k}=\breve{\theta}_{j,\vec k}^{i}+\eps_{j,\vec k}^{i}+\frac{1}{2h(0)\sqrt{n}}z_{j,\vec k}^{i}+\xi_{j,\vec k}^{i}
\end{equation}
where $\breve{\theta}_{j,\vec k}^{i}$ are the discrete wavelet coefficients of $g\left(\frac{\vec l}{T}\right)_{\vec 1\le \vec l \le \vec T}$ that are approximately equal to the true wavelet coefficients of $g$ $\theta_{j,\vec k}^{i}$, $\eps_{j,\vec k}^{i}$ are ``small" deterministic approximation errors, $z_{j,\vec k}^{i}$ are i.i.d $N(0,1)$, and $\xi_{j,\vec k}^{i}$ are ``small" stochastic errors. If it can be assumed that $\eps_{j,\vec k}^{i}$ and $\xi_{j,\vec k}^{i}$ are both negligible in some sense, we may be able to treat the model in the wavelet domain as an idealized sequence model
\[
y_{j,\vec k}^{i}\approx \breve{\theta}_{j,\vec k}^{i}+\frac{1}{2h(0)\sqrt{n}}z_{j,\vec k}^{i}
\] 
where $\frac{1}{2h(0)\sqrt{n}}$ plays the role of the noise level.  At this point, we can define a simple estimation procedure for the function $f$. Some auxiliary results are necessary before stating the main result. The first of these results is needed in order to bound the difference between the true wavelet coefficients $\theta^{i}_{j,\vec k}$ of the function $f\left(\frac{\vec l}{T}\right)_{\vec 1 \le \vec l \le \vec T}$ and discrete wavelet coefficients $\breve{\theta}^{i}_{j,k}$. Its proof is straightforward and is, therefore, omitted. 
\begin{lem}\label{bound_dif}
Let $T=2^{J}$ and define $f_{J}(u)=\frac{1}{\sqrt{V}}\sum_{\vec l\le \vec T}\sum_{i=1}^{2^{q}-1}f\left(\frac{\vec l}{T}\right)\psi_{J,\vec l}^{i}(u)$. Then, 
\[
\sup_{f\in B^{\alpha}_{s,t}(M)}\vert|f_{J}-f\vert|^{2}_{2}\le C\cdot q^{d}T^{-2d}
\]
where $d=\min\left(\alpha-\frac{q}{s},1\right)$. Moreover, $|\breve \theta_{j_{0},\vec k}^{i}-\theta_{j,\vec k}^{i}|\le CT^{-d}2^{-j/2}$ and so $\sum_{j=j_{0}}^{J-1}\sum_{\vec 1 \le k \le {\vec 2^{j}}}(\breve \theta_{j,\vec k}^{i}-\theta_{j,\vec k}^{i})^{2}\le CT^{-2d}$. 
\end{lem} 
 Another result that we need is the following proposition that studies the risk of our proposed procedure. 
 \begin{prop}\label{proporacle}
 Let the empirical wavelet coefficients $y_{j,\vec k}^{i}=\breve{\theta}_{j,\vec k}^{i}+\eps_{j,\vec k}^{i}+\frac{1}{2h(0)\sqrt{n}}z_{j,\vec k}^{i}+\xi_{j,\vec k}^{i}$ be as given in \eqref{emp_wav} and let estimated block thresholding coefficients $\hat \theta_{j,\vec k}^{i}$ be as defined in  \eqref{block_s}. Then, for some constant $C>0$, 
 \begin{align}\label{or_ineq}
 &\bE\,\sum_{\vec k\in B_{j,u}^{i}}(\hat \theta_{j,\vec k}^{i}-\breve{\theta}_{j,\vec k}^{i})^{2}\le \\
 &\le \min \left\{4\sum_{\vec k\in B^{i}_{j,u}}([\breve{\theta}_{j,\vec k}^{i}]^{2},8\lambda L n^{-1}\right\}+6\sum_{\vec k\in B^{i}_{j,u}}[\eps_{j,\vec k}^{i}]^{2}+C L n^{-2}\nonumber;
 \end{align}
 also, for any $0<\tau<1$, there exists a constant $C_{\tau}>0$ that depends on $\tau$ only such that for all $\vec k\in B^{i}_{j,u}$
 \[
 \bE\,(\hat \theta_{j,\vec k}^{i}-\breve{\theta}_{j,\vec k}^{i})^{2}\le C_{\tau}\min \left\{\max_{\vec k \in B^{i}_{j,u}}\{(\breve{\theta}_{j,\vec k}^{i}+\eps_{j,\vec k}^{i})^{2}\},Ln^{-1}\right\}+n^{-2+\tau}.
 \]
  \end{prop}
 \begin{proof}
 This proof is similar to the proof of Proposition 2 in \cite{brown2008robust}; therefore, we only give a brief proof for the inequality \eqref{or_ineq}.  First, we recall that $\vert \eps_{\vec l} \vert \le C\sqrt{\kappa}q^{d/2}T^{-d}$. The discrete wavelet transform of $\frac{\eps_{\vec l}}{\sqrt{n}}$ in our case is equal to $\eps^{i}_{j,\vec k}=\sum_{\vec l\in \bZ^{q}}\frac{\eps_{\vec l}}{\sqrt{n}}\int \phi_{j,\vec l}\psi_{j,\vec k}^{i}$. The Proposition \eqref{med_exp} suggests that $\sum_{j}\sum_{\vec k}\sum_{i}[\eps_{j,\vec k}^{i}]^{2}=\frac{1}{n}\sum_{\vec l \in \bZ^{q}}\eps_{\vec l}^{2}\le Cq^{d}T^{-2d}$ for some positive constant $C$ due to orthogonality of the discrete wavelet transform. Thus, we have for any $r>0$ 
 \[
 \bE\,\vert\xi_{j,\vec k}^{i}\vert^{r}\le C_{r}(\kappa n)^{-r/2}+C_{r}(\kappa n )^{-r/2}q^{dr/2}T^{-dr}
 \]
and, for any $a>0$
\begin{equation}\label{ineq}
P(\vert \xi_{j,\vec k}^{i}\vert>a)\le C_{r}^{'}(a^{2}\kappa n)^{-r/2}+C_{r}^{'}(a^{2}n T^{2d}/\kappa q^{d})^{-r/2}. 
\end{equation}
where $C_{r}$ and $C_{r}^{'}$ are constants that do not depend on $n$. 
At this point, we need to use Lemma $2$ of \cite{brown2008robust} with the number of observations being the size of each block $L$. For an $i$th wavelet function, $i=1,\ldots,2^{q}-1$, we have the expectation of the risk over each block bounded as 
\begin{align*}
&\bE\,\sum_{\vec k\in B^{i}_{j,u}}(\hat \theta_{j,\vec k}^{i}-\breve{\theta}_{j,k}^{i})^{2}\le \\
&\le \min \{4\sum_{\vec k\in B^{i}_{j,u}}(\breve{\theta}_{j,\vec k}^{i})^{2},8\lambda Ln^{-1}\}+6\sum_{\vec k\in B^{i}_{j,u}}(\eps_{j,\vec k}^{i})^2+\\
&+ 2n^{-1}\bE\sum_{\vec k\in B^{i}_{j,u}}(z_{j,\vec k}^{i}+2\sqrt{n}\xi_{j,\vec k}^{i})^{2}I\left(\sum_{\vec k\in B^{i}_{j,u}}(z_{j,\vec k}^{i}+2\sqrt{n}\xi_{j,\vec k}^{i})^{2}>\lambda L\right).
\end{align*}
Denote by $A$ the event that all $\vert \xi_{j,\vec k}^{i}\vert$ are bounded by $\frac{1}{2\sqrt{n}L}$, that is
\[
A=\{\vert2\sqrt{n}\xi_{j,\vec k}^{i}\vert \le L^{-1} \mbox{ for all } \vec k \in B^{i}_{j,u}\}.
\] 
Then it follows from \eqref{ineq} that for any $r\ge 1$, the probability of a complement of $A$ is
\begin{align*}\label{prob_compl}
&P(A^{'})\le \sum_{\vec k\in B^{i}_{j,u}}P\left(\vert 2\sqrt{n}\xi_{j,\vec k}^{i}\vert >L^{-1}\right)\\
&\le C_{r}^{'}(L^{-2}\kappa)^{-r/2}+C_{r}^{'}(L^{-2}T^{d}/\kappa q^{d})^{-r/2}.
\end{align*}

Thus, we have 
\begin{align*}
&D=\bE\,\sum_{\vec k\in B^{i}_{j,u}}(z_{j,\vec k}^{i}+2\sqrt{n}\xi_{j,\vec k}^{i})^{2}I\left(\sum_{\vec k\in B^{i}_{j,u}}(z_{j,\vec k}^{i}+2\sqrt{n}\xi_{j,\vec k}^{i})^{2}>\lambda L\right)\\
&=\bE\sum_{\vec k\in B^{i}_{j,u}}(z_{j,\vec k}^{i}+2\sqrt{n}\xi_{j,\vec k}^{i})^{2}I\left(A\cap\sum_{\vec k\in B^{i}_{j,u}}(z_{j,k}^{i}+2\sqrt{n}\xi_{j,k}^{i})^{2}>\lambda L\right)\\
&+\bE\sum_{\vec k\in B^{i}_{j,u}}(z_{j,\vec k}^{i}+2\sqrt{n}\xi_{j,\vec k}^{i})^{2}I\left(A^{c}\cap\sum_{\vec k\in B^{i}_{j,u}}(z_{j,\vec k}^{i}+2\sqrt{n}\xi_{j,\vec k}^{i})^{2}>\lambda L\right)\equiv D_{1}+D_{2}. 
\end{align*}
Recall that $(x+y)^{2}\le 2x^{2}+2y^{2}$ for any $x$ and $y$. At the next step, we have to use the inequality from Lemma $3$ of \cite{brown2008robust} (with the value $\tilde \lambda=\frac{\lambda L-\lambda-1}{L}$), and H\"{o}lder's inequality to obtain 
\begin{align*}
&D_{1}=\bE\,\sum_{\vec k\in B^{i}_{j,u}}(z_{j,\vec k}^{i}+2\sqrt{n}\xi_{j,\vec k}^{i})^{2}I\left(A\cap\sum_{\vec k\in B^{i}_{j,u}}(z_{j,\vec k}^{i}+2\sqrt{n}\xi_{j,\vec k}^{i})^{2}>\lambda L\right)\\
&\le 2\bE\,\sum_{\vec k\in B^{i}_{j,u}}[z_{j,\vec k}^{i}]^{2}I\left(\sum_{\vec k\in B^{i}_{j,u}}[z_{j,\vec k}^{i}]^{2}>\lambda L-\lambda-1\right)\\
&+8n\bE\,\sum_{\vec k\in B^{i}_{j,u}}[\xi_{j,\vec k}^{i}]^{2}I\left(\sum_{\vec k\in B^{i}_{j,u}}[z_{j,\vec k}^{i}]^{2}>\lambda L-\lambda-1\right)\\
&\le 2(\lambda L-\lambda-1)e^{-L/2(\lambda-(\lambda+1)L^{-1}-\log (\lambda-(\lambda+1)L^{-1})-1)}\\
&+8n\sum_{\vec k\in B^{i}_{j,u}}(\bE\,[\xi_{j,\vec k}^{i}]^{2v})^{1/v}\left(P\left(\sum_{\vec k\in B^{i}_{j,u}}[z_{j,\vec k}^{i}]^{2}>\lambda L-\lambda-1\right)\right)^{1/\omega}
\end{align*}
where $v,\omega>1$ and $\frac{1}{v}+\frac{1}{\omega}=1$. Recall that $\kappa=n^{1/4}$ and choose $\frac{1}{\omega}=1-\frac{1}{4}=\frac{3}{4}$. This lets us conclude that $D_{1}\le CLn^{-1}$. Arguing similarly, we conclude that $D_{2}\le n^{-1}$ and so the final inequality is obtained.
 \end{proof}
\begin{rem}
Note that the tail probability $P(\vert \xi_{j,\vec k}^{i}\vert>a)$ must decay faster than any polynomial in $n$ to ensure that the contribution of $\xi_{j,\vec k}^{i}$ to the squared risk of the proposed procedure is negligible compared to that of $z_{j,\vec k}^{i}$. Recall that $\kappa=n^{1/4}$. Then, $\frac{T^{2d}}{\kappa}=n^{\frac{6d-q}{4q}}$ and we have to require that $6d-q>0$, or $d=\min\left(\alpha-\frac{q}{s},1\right)>\frac{q}{6}$. Since $d$ characterizes smoothness of the H\"{o}lder ball the Besov ball is embedded into, it may be a good idea to describe this requirement in terms of the original smoothness indicator $\alpha$. Note that (see Remark \eqref{cond}) that, due to approximation error over multivariate Besov spaces, we must have $\frac{3d}{2q}> \frac{2\alpha}{2\alpha+q}$. To guarantee that $d>\frac{q}{6}$, we may require that $ \frac{4\alpha q}{3(2\alpha+q)}> \frac{q}{6}$ which is equivalent to $\alpha> \frac{q}{6}$. This is the origin of the lower bound on $\alpha$ in the statement of Theorems \eqref{gl_adapt} and \eqref{lc_adapt}. 
\end{rem}   

{\bf Proof of the Theorem \eqref{gl_adapt}}
First, note that 
\[
\bE\,\|\hat f_{n}-f\|^{2}_{2}\le 2\bE\,\|\hat g_{n}-g\|^{2}_{2}+2\bE\,(\hat b_{\vec l}^{2}-b_{\vec l})^{2}. 
\]
By selecting $\kappa\asymp n^{1/4}$, we ensure that $\bE\,(\hat b_{\vec l}^{2}-b_{\vec l})^{2}=o(n^{-2\alpha/2\alpha+q})$ and so we need only to focus on bounding $\bE\,\|\hat g_{n}-g\|^{2}_{2}$. Note that, if the intercept $a$ and/or the median of the vector $X_{\vec i}$ in the model \eqref{model1} are non-zero, an appropriate term in the model can be estimated at the rate of $n^{-1}=o(n^{-2\alpha/2\alpha+q})$. Using the notation of Section \eqref{intro}, let $A=a+(med X_{\vec i})^{'}\beta$ and $A_{n}$ an asymptotically normal $\sqrt{n}$ convergent estimator of $A$. Such an estimator can be easily obtained as in \cite{brown2016semiparametric}. In that case we will have $\bE\,\|\hat f_{n}-f\|^{2}_{2}\le 2\bE\,\|\hat g_{n}-g\|^{2}_{2}+2\bE\,(\hat b_{\vec l}^{2}-b_{\vec l})^{2}+2E\,(\hat A_{n}-A)^{2}$ where $E\,(\hat A_{n}-A)^{2}=o(n^{-2\alpha/2\alpha+q})$. Since functions $f$ and $g$ only differ by a constant $b_{\vec l}$, their wavelet coefficients coincide, that is, $\theta_{j,\vec k}^{i}=\int_{[0,1]^{q}} f\vec \psi_{j,\vec k}^{i}=\int_{[0,1]^{q}]}g\psi_{j,\vec k}^{i}$. To make our analysis straightforward, we expand $ \bE\,\|\hat g_{n}-g\|^{2}_{2}$ as follows:
\begin{align*}
&\bE\,\| \hat g_{n}^{2}-g\|^{2}_{2}=\sum_{\vec 1 \le \vec k\le \vec {2^{j_0}}}\sum_{i=1}^{2^{q}-1} \bE\,[\hat{\theta}_{j_{0},\vec k}-\theta_{j_{0},\vec k}]^{2}+\sum_{j=j_{0}}^{J-1}\sum_{\vec 1 \le \vec k\le \vec{2^{j}}} \sum_{i=1}^{2^{q}-1}\bE\,(\hat {\theta}_{j,\vec k}^{i}-\theta_{j,\vec k}^{i})^{2}\\
&+\sum_{j=J}^{\infty}\sum_{\vec 1 \le \vec k\le \vec {2^{j}}}\sum_{i=1}^{2^{q}-1}(\theta_{j,\vec k}^{i} )^{2}\equiv S_{1}+S_{2}+S_{3}.
\end{align*}
First, we note that the term $S_{1}$ is asymptotically small. Indeed, by definition, $\hat \theta_{j_{0},\vec k}=y_{j_{0},\vec k}$. Since  $\hat \theta_{j_{0},\vec k}-\theta_{j_{0},\vec k}=(y_{j_{0},\vec k}-\breve \theta_{j_{0},\vec k})+(\breve \theta_{j_{0},\vec k}-\theta_{j_{0},\vec k})$, we have
\[
S_{1}\le C\cdot 2^{j_{0}q}n^{-1}\eps^{2}+CT^{-2d}=o(n^{-2\alpha/2\alpha+q}).
\]
The term $S_{3}$ is also small asymptotically. To show that this is true, we note first that $2^{j\left(\alpha+q\left(\frac{1}{2}-\frac{1}{s}\right)\right)}\left(\sum_{\vec 1 \le \vec k\le \vec {2^{j}}}\sum_{i=1}^{2^{q}-1}|\theta_{j,k}^{i}|^s\right)^{1/s}\le M$ for any function $f\in B^{\alpha}_{s,t}(M)$. Then, using the inequality $\vert|x\vert|_{p_{2}}\le \vert|x\vert|_{p_{1}}\le q^{1/p_{1}-1/p_{2}}\vert|x\vert|_{p_{2}}$ for any $0<p_{1}\le p_{2}\le \infty$ and $x\in \bR^{q}$, we obtain that
\[
S_{3}\le 2^{-2J\min\left(\alpha,\alpha+q\left(\frac{1}{2}-\frac{1}{s}\right)\right)}=o(n^{-2\alpha/2\alpha+q})
\]
due to assumptions on $J$ and $\alpha$. In the next step, we will use the Proposition \eqref{proporacle} to analyze the term $S_{2}$. Next, we find out that 
\begin{align*}\label{S2_tr}
&S_{2}\le 2\sum_{j=j_{0}}^{J-1}\sum_{\vec 1 \le \vec k\le \vec {2^{j}}}\sum_{i=1}^{2^{q}-1}\bE\,(\hat{\theta}^{i}_{j,\vec k}-\breve{\theta}_{j,\vec k}^{i})^{2}+2\sum_{j=j_{0}}^{J-1}\sum_{i=1}^{2^{q}-1}(\breve{\theta}_{j,\vec k}^{i}-\theta^{i}_{j,\vec k})^{2}\\
&\le \sum_{j=1}^{J-1}\sum_{u=1}^{2^{qj}/L}\sum_{i=1}^{2^{q}-1}\min\left\{8\sum_{\vec k\in B^{i}_{j,u}}[\breve{\theta}_{j,\vec k}^{i}]^{2},8\lambda^{*}Ln^{-1}\right\}+6\sum_{j=j_{0}}^{J-1}\sum_{\vec 1\le \vec k\le \vec {2^{j}}}\sum_{i=1}^{2^{q}-1}[\eps_{j,\vec k}^{i}]^{2}\\
&+Cn^{-1}+2\sum_{j=j_{0}}^{J-1}\sum_{\vec 1 \le \vec k\le \vec {2^{j}}}\sum_{i=1}^{2^{q}-1}[\breve{\theta}_{j,\vec k}^{i}-\theta_{j,\vec k}^{i}]^{2}\\
&\le \sum_{j=1}^{J-1}\sum_{u=1}^{2^{qj}/L}\sum_{i=1}^{2^{q}-1}\min\left\{8\sum_{\vec k\in B^{i}_{j,u}}[\breve{\theta}_{j,\vec k}^{i}]^{2},8\lambda^{*}Ln^{-1}\right\}+Cn^{-1}+CT^{-2d}.
\end{align*}
At this point, we consider two different cases. First, start with $p\ge 2$. Select $J_{1}=\left \lfloor \frac{q}{2\alpha+q}\log_{2}n\right\rfloor$ which implies that $2^{J_{1}}\approx n^{q/2\alpha+q}$. Thus, using the result of Lemma \eqref{bound_dif}. we obtain
\[
S_{2}\le 8\lambda^{*} \sum_{j=j_{0}}^{J_{1}-1}\sum_{u=1}^{2^{qj}/L}\sum_{i=1}^{2^{q}-1}Ln^{-1}+8\sum_{j=J_{1}}^{J-1}\sum_{\vec 1 \le \vec k \le \vec {2^{j}}}\sum_{i=1}^{2^{q}-1}[\breve{\theta}_{j,\vec k}^{i}]^{2}+Cn^{-1}+CT^{-2d}\le Cn^{-2\alpha/2\alpha+q}.
\]
Next, consider the case of $p<2$. First, note that 
\begin{align*}
&\sum_{u=1}^{2^{jq}/L}\sum_{i=1}^{2^{q}-1}\left(\sum_{\vec k\in B^{i}_{j,u}}[\theta_{j,\vec k}^{i}]^{2}\right)^{s/2}\\
&\le \sum_{\vec 1 \le \vec k\le 2^{\vec j}}([\theta_{j,\vec k}^{i}]^{2})^{s/2}\le M2^{-jws}.
\end{align*}
Select $J_{2}$ such that $2^{J_{2}}\asymp n^{1/2\alpha+q}(\log n)^{2-s/s(2\alpha+q)+2(1-q)}$. Using Lemma 6 from \cite{brown2008robust}, one obtains 
\begin{align*}
&\sum_{j=J_{2}}^{J-1}\sum_{u=1}^{2^{j}/L}\min\left\{8\sum_{\vec k\in B^{i}_{j,u}}[\breve{\theta}_{j,\vec k}^{i}]^{2},8\lambda^{*}Ln^{-1}\right\}\\
&\le Cn^{-2\alpha/2\alpha+q}(\log n)^{2-s/s(2\alpha+q)+2(1-q)}. 
\end{align*}
On the other hand, we also have 
\begin{align*}
&\sum_{j=j_{0}}^{J_{2}-1}\sum_{u=1}^{2^{qj}/L}\min\left\{8\sum_{\vec k\in B^{i}_{j,u}}[\breve{\theta}_{j,\vec k}^{i}]^{2},8\lambda^{*}Ln^{-1}\right\} \\
&\sum_{j=j_{0}}^{J_{2}-1}\sum_{u=1}^{2^{qj}/L}8\lambda_{*}Ln^{-1}\le Cn^{-2\alpha/2\alpha+q}(\log n)^{2-s/[s(2\alpha+q)+2(1-q)]}.
\end{align*}
Thus, we can now confirm that the $L_{2}$ risk in the case $p<2$ is bounded from above uniformly as
\[
\bE\,\| \hat f_{n}-f\|^{2}_{2}\le Cn^{-2\alpha/2\alpha+q}(\log n)^{2-s/[s(2\alpha+q)+2(1-q)]}
\]
\begin{rem}\label{cond}
In order to ensure that the risk of $\hat b_{\vec l}$ is negligible, we need to have $\kappa^{-4}=o(n^{-2\alpha/2\alpha+q})$; note that $\kappa=n^{1/4}$ satisfies this assumption. Also, to make the approximation error $\|f_{J}-f\|^{2}_{2}$ negligible, we need to have $T^{-2d}=O(n^{-2\alpha/2\alpha+q})$. It is easy to see that this is guaranteed by the inequality $\frac{3d}{2q}>\frac{2\alpha}{2\alpha+q}$. Note that the latter, rather ponderous, assumption, is needed due to approximation over the $q$-dimensional Besov spaces.
\end{rem}

{\bf Proof of Theorem \eqref{lc_adapt}}
As in the proof of the Theorem \eqref{gl_adapt}, and without loss of generality, we can assume that the $med(X_{i})$ is identically equal to zero; if this is not the case, an additional term can be estimated at the rate of $n^{-1}=o\left(\frac{\log n}{n}\right)^{\alpha/2\alpha+q}$. Next, note that for all $f\in \Lambda^{\alpha}(M)$, the absolute values of its wavelet coefficients are $|\theta_{j,\vec k}^{i}|=|\langle f, \psi_{j,\vec k}^{i}\rangle |\le C2^{-j(q/2+\alpha)}$ for some constant $C>0$ that does not depend on $f$. Also, note that for any random variables $X_{i}$, $i=1,\ldots,n$, $\bE\,\left(\sum_{i=1}^{n}X_{i}\right)^{2}\le \left(\sum_{i=1}^{n}(\bE\,X_{i}^{2})^{1/2}\right)^{2}$. Then, we have 
\begin{align*}
&\bE\,(\hat f_{n}(\vec u_{0})-f(\vec u_{0}))^{2}\\
&=\bE\,\Big[\sum_{\vec 1 \le \vec k\le \vec 2^{j_{0}}}(\hat \theta_{j_{0},\vec k}-\theta_{j_{0},\vec k})\phi_{j_0,\vec k}^{i}(\vec u_{0})
+\sum_{j=j_{0}}^{\infty}\sum_{\vec 1 \le \vec k\le \vec 2^{j}}\sum_{i=1}^{2^{q}-1}(\hat \theta_{j,\vec k}^{i}-\theta_{j,\vec k}^{i})\psi_{j,\vec k}^{(i)})(\vec u_{0})\\
&-(\hat b_{\vec l}-b_{\vec l})\Big]^{2}\\
&\le \Big[(\bE\,(\hat b_{\vec l}-b_{\vec l})^{2})^{1/2}+\sum_{\vec 1 \vec k\le \vec 2^{j_{0}}}(\bE\,(\hat \theta_{j_{0},\vec k}-\theta_{j_{0},\vec k})^{2}\phi_{j_{0},\vec k}^{2}(\vec u_{0}))^{1/2}\\
&+\sum_{j=j_{0}}^{J-1}\sum_{\vec 1 \le \vec k\le \vec 2^{j}}\sum_{i=1}^{2^{q}-1}(\bE\,(\hat \theta_{j,\vec k}^{i}-\theta_{j,\vec k}^{i})^{2}[\psi_{j,\vec k}^{i}]^{2}(\vec u_{0}))^{1/2}+\sum_{j=J}^{\infty}\sum_{\vec 1 \le \vec k\le \vec 2^{j}}\sum_{i=1}^{2^{q}-1}|\theta_{j,\vec k}^{i}\psi_{j,\vec k}^{i}(\vec u_{0})|\Big]^{2}\\
&\equiv (Q_{1}+Q_{2}+Q_{3}+Q_{4})^{2}.
\end{align*}
First of all, we note that, due to Proposition \eqref{med_prop}, we have $Q_{1} =(\bE\,(\hat b_{\vec l}-b_{\vec l})^{2})^{1/2}=o(n^{-\alpha/2\alpha+q})$. Next, clearly we have $Q_{2}=\sum_{\vec 1 \le \vec k\le \vec 2^{j_{0}}}(\bE\,(\hat \theta_{j_{0},\vec k}-\theta_{j_{0},\vec k})^{2}\vert\phi_{j_{0},\vec k}(\vec u_{0})\vert=O(n^{-1})$.  Recall that for any sequence of translated and rescaled $i$th wavelet $\psi_{j,\vec k}^{i}$ there are at most $N$ that are nonvanishing at the point $\vec u_{0}$; here, $N$ is the length of support of $\psi^{i}$. Mathematically, we have $K(t_{0},j)=\{\vec k:\psi_{j,\vec k}^{(i)}(\vec u_{0})\ne 0\}$ such that $\vert K_{t_{0},j} \vert \le N$.  
Thus, we have 
\begin{equation}\label{q4}
Q_{4}=\sum_{j=J}^{\infty}\sum_{\vec 1 \le \vec k\le \vec 2^{j}}\sum_{i=1}^{2^{q}-1}|\theta_{j,\vec k}^{i}||\psi_{j,\vec k}^{i}(\vec u_{0})|\le \sum_{j=J}^{\infty}N 2^{q}\|\psi\|_{\infty}2^{jq/2}C2^{-j(q/2+\alpha)}\le C T^{-\alpha}.
\end{equation}
Finally, if we select a sufficiently small $\tau$, and use the second inequality from the Proposition \eqref{med_prop}, we have 
\begin{align*}\label{q3}
&Q_{3}\le \sum_{j=j_{0}}^{J-1}\sum_{K(t_{0},j)}\sum_{i=1}^{2^{q}-1}2^{jq/2}\|\psi\|_{\infty}(\bE\,(\hat \theta_{j,\vec k}^{i}-\theta_{j,\vec k}^{i})^{2})^{1/2}\\
&\le C\sum_{j=j_{0}}^{J-1}2^{jq/2}\left[\min(2^{-j(q+2\alpha)}+T^{-2\alpha \wedge 1}2^{-jq},Ln^{-1})+n^{-2+\tau}\right]^{1/2}\\ &\le  C\left(\frac{\log n}{n}\right)^{\alpha/2\alpha+q}.\nonumber
\end{align*}
The final statement of this theorem is then obtained combining all inequalities for $Q_{1}$, $Q_{2}$, $Q_{3}$, and $Q_{4}.$

\bibliographystyle{apalike} 
\bibliography{reference}
\end{document}